\documentclass[opre,nonblindrev]{informs3opt}
\usepackage{color,soul}
\usepackage{palatino,url}
\usepackage{verbatim, amssymb}
\usepackage{amsmath,mathrsfs,thmtools}
\usepackage{graphics,amsopn,amstext,amsfonts,color}
\usepackage{bm}
\usepackage{setspace}
\usepackage{subfigure}
\usepackage{graphicx}
\usepackage{epstopdf}
\usepackage{natbib}

\usepackage{multirow}
\usepackage{booktabs}
\usepackage{paralist}
\usepackage{colortbl}
\usepackage{titlesec}
\usepackage[titletoc,toc,title]{appendix}

\usepackage{xcolor}
\usepackage{cases}
\usepackage{empheq}
\usepackage[hidelinks,bookmarksnumbered=true]{hyperref}
\usepackage{graphics,latexsym,amsfonts,verbatim,amsmath}

\usepackage{algorithm}
\usepackage{algpseudocode}
\usepackage{todonotes}
\usepackage{threeparttable}

\usepackage{empheq}
\usepackage{cases}
\usepackage{enumitem}

\def\E{{\mathbb E}}

\def\Pr{{\mathbb{P}}}

\def\Re{\mathbb{R}}
\def\Qe{\mathbb{Q}}
\def\I{\mathbb{I}}

\def\hat{\widehat}

\def \bxi{\bm{\xi}}
\def \bzeta{\tilde{\bm{\zeta}}}
\def \bfzeta{\bm{\zeta}}

\def \P{\mathcal{P}}

\def \Ze{{\mathbb{Z}}}
\def \Z{{\mathcal{Z}}}

\def\P{{\mathcal P}}
\def\W{{\bm W}}
\def\X{{\mathcal X}}

\def\C{{\mathcal C}}

\def\Re{{\mathbb R}}

\newcommand{\exclude}[1]{}
\newcommand{\rxi}{\bm{\xi}}
\newcommand{\bfx}{\bm{x}}  

\newcommand{\bfy}{\bm{y}}

\newcommand{\e}{\bm{e}}

\newcommand{\q}{\bm{q}}

\newcommand{\T}{\bm{T}}

\def\CVaR{{\bf{CVaR}}}

\newcommand{\zsetd}{Z}

\DeclareMathOperator{\conv}{conv}

\newcommand*{\qedA}{\hfill\ensuremath{\square}}

\usepackage{thmtools}
\usepackage{thm-restate}

\usepackage{cleveref}

\declaretheorem[name=Theorem]{theorem}
\declaretheorem[name=Proposition]{proposition}
\declaretheorem[name=Lemma]{lemma}

\declaretheorem[name=Example]{example}

\declaretheorem[name=Property]{property}

\newcommand*{\QEDA}{\hfill\ensuremath{\square}}

\usepackage{natbib}
 \bibpunct[, ]{(}{)}{,}{a}{}{,}%
 \def\bibfont{\small}%
\allowdisplaybreaks
\begin{document}
\RUNAUTHOR{Weijun Xie}
\RUNTITLE{Tractable Reformulations of DRTSP with $\infty-$Wasserstein Distance
}

\TITLE{Tractable Reformulations of Distributionally Robust Two-stage Stochastic Programs with $\infty-$Wasserstein Distance}

\ARTICLEAUTHORS{%
\AUTHOR{Weijun Xie}
\AFF{Department of Industrial \& Systems Engineering, Virginia Tech, Blacksburg, VA 24061, \EMAIL{wxie@vt.edu.}}

}


\ABSTRACT{
In the optimization under uncertainty, decision-makers first select a wait-and-see policy before any realization of uncertainty and then place a here-and-now decision after the uncertainty has been observed. Two-stage stochastic programming is a popular modeling paradigm for the optimization under uncertainty that the decision-makers first specifies a probability distribution, and then seek the best decisions to jointly optimize the deterministic wait-and-see and expected here-and-now costs. In practice, such a probability distribution may not be fully available but is probably observable through an empirical dataset. Therefore, this paper studies distributionally robust two-stage stochastic program (DRTSP) which jointly optimizes the deterministic wait-and-see and worst-case expected here-and-now costs, and the probability distribution comes from a family of distributions which are centered at the empirical distribution using $\infty-$Wasserstein metric. There have been successful developments on deriving tractable approximations of the worst-case expected here-and-now cost in DRTSP. Unfortunately, limited results on exact tractable reformulations of DRTSP. This paper fills this gap by providing sufficient conditions under which the worst-case expected here-and-now cost in DRTSP can be efficiently computed via a tractable convex program. By exploring the properties of binary variables, the developed reformulation techniques are extended to DRTSP with binary random parameters. The main tractable reformulations in this paper are projected into the original decision space and thus can be interpreted as conventional two-stage stochastic programs under discrete support with extra penalty terms enforcing the robustness. These tractable results are further demonstrated to be sharp through complexity analysis.
 }

\KEYWORDS{Distributionally Robust, Two-stage, Stochastic Program,  Tractable, Reformulation.} \HISTORY{}

\maketitle

\newpage

\newpage
\section{Introduction}
\subsection{Setting}
Consider the distributionally robust two-stage stochastic program (DRTSP) of the form \citep{hanasusanto2016conic}:
\begin{subequations}\label{sp}
\begin{align}
v^*= \min_{\bm{x}} & \ \ \bm{c}^{\top}\bm{x}+\Z(\bm{x}), \label{sp-obj} \\
\rm{s.t. } & \ \ \bm{x} \in \X, \label{sp-det}\\
& \ \ \Z(\bm{x})=\sup_{\Pr\in \P}\E_{\Pr}[Z(\bm{x},\tilde{\bm{\xi}})].\label{sp-recourse}
\end{align}
\end{subequations}
Above, set $\X\subseteq \Re^n$ denotes the feasible region of the here-and-now decisions $\bm{x}$, the vector $\bm c \in \Re^{n_1}$ denotes the here-and-now objective coefficients, and the function $\Z(\bm{x})$ denotes the worst-case expected wait-and-see cost $Z(\bm{x},\tilde{\bm{\xi}})$ (also known as, recourse function) specified by random parameters $\tilde{\bm{\xi}}\in \Xi$, where its probability distribution $\Pr$ comes from a family of distributions, denoted by ambiguity set $\P$. 

Following the notation in \cite{ahmed2010two,bertsimas2010models,Shapiro2009lectures}, given a realization $\bm{\xi}$ of $\tilde{\bm{\xi}}$, we consider the following recourse function:
\begin{subequations}\label{recourse}
\begin{align}
Z(\bm{x},\bm{\xi})=\min_{\bm{y}} & \ \ (\bm{Q}{\bm{\xi}}_q+\bm{q})^\top \bm{y}, \label{sp-obj_2nd} \\
\rm{s.t. } & \ \ \bm{T}(\bm{x}){\bm{\xi}}_T+\bm{W} \bm {y}\geq \bm{h}(\bm{x}), \label{sp-det_2nd}\\
& \ \ \bm{y}\in \Re^{n_2},\label{sp-bound_2nd}
\end{align}
\end{subequations}
where $\bm{y}$ represents the second-stage wait-and-see decisions, $\bm\xi=(\bm{\xi}_q, \bm{\xi}_T)\in \Re^{m_1}\times \Re^{m_2}$, $\bm{q}\in \Re^{n_2}$ and there are two affine mappings- right-hand mapping $\bm{h}: \Re^{n_1}\rightarrow \Re^\ell$ and  technology mapping $\bm{T}: \Re^{\ell\times m_2}\rightarrow \Re^{\ell}$. Similar to many two-stage stochastic program \citep{bertsimas2010models,Shapiro2009lectures}, throughout this paper, we assume that 
\begin{itemize}
\item (Fixed Recourse) The recourse matrix $\bm{W}\in \Re^{\ell\times n_2}$ is \textit{fixed}; and
\item (Separable Uncertainty) The support $\Xi=\Xi_q\times \Xi_T$, where $\Xi_q\subseteq \Re^{\ell},\Xi_T\subseteq \Re^{n_2}$.
\end{itemize}
Both assumptions are quite standard and have appeared in many stochastic programming applications, for example, power systems \citep{dai2013trading,golari2014two}, logistics and supply chain \citep{kara2010stochastic,lu2015reliable}, inventory and production \citep{hu2016two,zhang2018multi}, agriculture \citep{li2013inexact}, and many others.

In this paper, we consider $\infty-$Wasserstein ambiguity set $\P$, which is defined as
	\begin{align}\label{eq_general_das}
\P=\left\{\Pr:\Pr\left\{\tilde{\bxi}\in \Xi\right\}=1,W^\infty\left(\Pr,\Pr_{\bzeta}\right)\leq \theta\right\},
\end{align}
where $\infty-$Wasserstein distance \citep{bertsimas2018data,givens1984class} is defined as
\[W^\infty\left(\Pr_1,\Pr_{2}\right)=\inf_{\Qe}\left\{\text{ess.sup}\|{\bm{\xi}}_1-{\bm{\xi}}_2\|_p\Qe(d\bm{\xi}_1,d\bm{\xi}_2):\begin{array}{l}\text{$\Qe$ is a joint distribution of $\tilde{\bm{\xi}}_1$ and $\tilde{\bm{\xi}}_2$}\\
\text{with marginals $\Pr_1$ and $\Pr_2$, respectively}\end{array}\right\},\]
\text{ess.sup}$(\cdot)$ denotes essential supremum (see \citealt{rudin1964principles}), norm $\|\cdot\|_p$ denotes reference distance with $p\in [1,\infty]$ and $\Pr_{\bzeta}$ denotes a discrete empirical distribution of $\bzeta$ generated by i.i.d. samples $\Z=\{\bm{\zeta}^j:=(\bm{\zeta}_q^j,\bm{\zeta}_T^j)\}_{j\in [N]}\subseteq \Xi$ from the true distribution $\Pr^{\infty}$, i.e., its point mass function is $\Pr_{\bzeta}\left\{\bzeta=\bm{\zeta}^j\right\}=\frac{1}{N}$, $\theta\geq0$ denotes the Wasserstein radius, and $p\geq 1$. Many recent works also studied $\tau-$Wasserstein ambiguity set with $\tau\in [1,\infty)$, where in \eqref{eq_general_das}, we replace the $\infty-$Wasserstein distance by the following $\tau-$Wasserstein distance
\[W^\tau\left(\Pr_1,\Pr_{2}\right)=\inf_{\Qe}\left\{\sqrt[\tau]{\int_{\Xi\times\Xi}\|{\bm{\xi}}_1-{\bm{\xi}}_2\|_p^\tau\Qe(d\bm{\xi}_1,d\bm{\xi}_2)}:\begin{array}{l}\text{$\Qe$ is a joint distribution of $\tilde{\bm{\xi}}_1$ and $\tilde{\bm{\xi}}_2$}\\
\text{with marginals $\Pr_1$ and $\Pr_2$, respectively}\end{array}\right\}.\]
Clearly, according to \cite{givens1984class}, $\tau-$Wasserstein distance converges to $\infty-$Wasserstein distance as $\tau\rightarrow \infty$. Different types of Wasserstein ambiguity set might provide different tractable results. The results of this paper reveal that $\infty-$Wasserstein ambiguity set indeed delivers more tractable results for DRTSP and according to \cite{bertsimas2018data}, it still exhibits attractive convergent properties.

The discussions on advantages of Wasserstein ambiguity sets can be found in \cite{esfahani2015data,gao2016distributionally,bertsimas2018data}, which are briefly summarized below: (i) \textbf{Data-Driven}. When the number of observed empirical data points grows, the Wasserstein radius shrinks under mild conditions, and thus, the corresponding DRTSP \eqref{sp} eventually converges to the true two-stage stochastic programming as $N\rightarrow \infty$; (ii) \textbf{Finite}. It has been shown in \cite{esfahani2015data,gao2016distributionally,blanchet2016quantifying} that as long as the number of empirical data points is finite, the worst-case probability distribution of the corresponding DRTSP \eqref{sp} is also finitely supported; and (iii) \textbf{Tractability}. There have been many successful developments on tractable reformulations of distributionally robust optimization with Wasserstein ambiguity set, see, for example, \cite{esfahani2015data,gao2016distributionally,blanchet2016quantifying,blanchet2016robust,gao2017wasserstein,chen2019sharing}. However, for DRTSP, the tractable results are quite limited. Therefore, this paper focuses on developing tractable representations of DRTSP under $\infty-$Wasserstein ambiguity set $\P$, in particular, the tractable representations of the worst-case expected wait-and-see cost (i.e., the function $\Z(\bm{x})$).

\subsection{Related Literature}

Distributionally robust optimization (DRO) has been an alternative modeling paradigm for optimization under uncertainty, where the probability distributions of random parameters are not fully known. Interested readers are referred to \cite{rahimian2019DROreview} for a complete literature review of DRO. Recently, there are several interesting works on exact tractable reformulations of the function $\Z(\bfx)$ under three types of ambiguity sets, namely, under moment ambiguity set, phi-divergence based ambiguity set, and Wasserstein ambiguity set. 
\begin{enumerate}[label=(\roman*),wide = 0pt, itemsep=1.5ex]
\item Moment ambiguity set is specified by the acquired knowledge of some moments (e.g., known first two moments), and has been successfully applied to many different settings (see for example, \citealt{delage2010distributionally,bertsimas2010models,goh2010distributionally,bertsimas2018adaptive,wiesemann2014distributionally,hanasusanto2015distributionally,hanasusanto2015Ambiguous,Natarajan2017,li2016ambiguous,xie2016opf,Xie2016drccp,zhang2016drccbp}). \cite{delage2010distributionally} shows that if the first two moments are known or bounded from above, and the recourse function can be expressed as piecewise maximum of a finite number of functions which are convex in $\bm{x}$ and concave in the random parameters $\tilde{\bxi}$, then the function $\Z(\bfx)$ have a tractable representation. In \cite{bertsimas2010models}, the authors showed that if first two moments are known, then the function $\Z(\bfx)$ with only objective uncertainty (i.e., $\tilde{\bxi}_T$ is deterministic) can be formulated as a tractable semidefinite program (SDP). \cite{Natarajan2017} further showed that if first two moments are known, then the function $\Z(\bfx)$ with objective uncertainty and any known support can be reformulated as an SDP, where the positive semidefinite matrix comes from a convex hull of rank-one matrices, and, although computationally intractable in general, the authors were able to establish sufficient conditions under which this SDP formulation becomes tractable.
\item Phi-divergence based ambiguity set is specified by the bounded distance between a nominal distribution and true distribution via phi-divergence \citep{bayraksan2015data,ben2013robust,hu2012kullback,jiang2013data,jiang2018risk}. In particular, \cite{jiang2018risk} showed that for DRTSP with phi-divergence based ambiguity set can be equivalently reformulated as a convex combination of conditional-value-at-risk and worst-case risk costs, where the tractability follows when both risk measures are tractable.

\item Wasserstein ambiguity set is specified by the bounded distance between a nominal distribution and true distribution via Wasserstein metric \citep{esfahani2015data,blanchet2016quantifying,blanchet2016robust,chen2018data,chen2019robust,chen2019sharing,gao2016distributionally,gao2017wasserstein,hanasusanto2016conic,bertsimas2018data,luo2017decomposition,xie2018drccp,Xie2018approx,zhao2015data_a}. \cite{esfahani2015data} showed that for DRTSP under $1-$Wasserstein ambiguity set, if the recourse function can be expressed as piecewise maximum of a finite number of functions which are bi-affine in the decision variables $\bm{x}$ and the random parameters $\tilde{\bxi}$, then the function $\Z(\bfx)$ has a tractable representation. 
\cite{hanasusanto2016conic} extended the tractable results into DRTSP with constraint uncertainty (i.e., $\tilde{\bxi}_q$ is deterministic) and $1-$Wasserstein ambiguity set, where the reference distance $\|\cdot\|_1$ and support $\Xi_T=\Re^{m_2}$, and proved that for general DRTSP under Wasserstein ambiguity set, it is in general NP-hard to evaluate the function $\Z(\bfx)$. Thereby, the authors proposed a hierarchy of SDP representations to approximate the function $\Z(\bfx)$ under $2-$Wasserstein ambiguity set. 

Different from  \cite{hanasusanto2016conic}, this paper focuses on $\infty-$Wasserstein ambiguity set, providing sufficient conditions under which the function $\Z(\bfx)$ can be tractable, even with both objective and constraint uncertainties, and further extending the tractable results to the cases where part of random parameters are binary. As far as the author is concerned, only two works studied $\infty-$Wasserstein ambiguity set, i.e., \cite{bertsimas2018data,bertsimas2019twostage}. \cite{bertsimas2018data} provided fundamental convergence analysis of $\infty-$Wasserstein ambiguity set, and studied adaptive approximation schemes for the data-driven multi-stage linear program, while \cite{bertsimas2019twostage} studied robust two-stage sampling problem with constraint uncertainty and proved that under certain conditions, the proposed multi-policy approximation scheme is asymptotically optimal. Different from these two works, this paper studies DRTSP by exploring exact tractable reformulations of the function $\Z(\bfx)$ with $\infty-$Wasserstein ambiguity set and providing the complexity analysis to demonstrate the sharpness of the tractable results.
\end{enumerate}

\subsection{Contributions}
This paper studies exact reformulations of the worst-case expected wait-and-see cost (i.e., function $\Z(\bfx)$) in distributionally robust two-stage stochastic program (DRTSP) under $\infty-$Wasserstein ambiguity set. The main contributions are highlighted as below.
\begin{enumerate}[label=(\roman*)]
\item When random parameters $(\tilde{\bm{\xi}}_q,\tilde{\bm{\xi}}_T )$ are continuous, we derive exact tractable reformulations for the function $\Z(\bfx)$ with uncertainties in both objective function and constraint system, with objective uncertainty only, as well as with constraint uncertainty only. We prove that our tractable results are sharp.

\item When either of random parameters $(\tilde{\bm{\xi}}_q,\tilde{\bm{\xi}}_T)$ are binary, by exploring the binary variables in the reformulation, we are able to derive exact tractable reformulations for the function $\Z(\bfx)$ under sufficient conditions. Our complexity results show that the tractable results are sharp.

\item The main tractable reformulations in this paper are projected to the original decision space, and thus have straightforward interpretations of robustness.


\item We demonstrate that if the conditions provided in above results do not hold, then the proposed reformulations become tractable upper bound and will become exact if the Wasserstein radius goes to zero, i.e., they are asymptotically optimal.
\end{enumerate}

The remainder of the paper is organized as follows. Section \ref{sec_prelim} introduces the preliminary results that will be used throughout the rest of this paper. \Cref{sec_cont} presents exact tractable {reformulations} of DRTSP with continuous random parameters. \Cref{sec_bin} extends the results for DRTSP with binary random parameters. The main results and recommendations are summarized in \Cref{sec_rec} and \Cref{sec_sep_numerical} numerically illustrates the proposed formulations. Section~\ref{sec_conclusion} concludes the paper.\\

\noindent {\em Notation:} The following notation is used throughout the paper. We use bold-letters (e.g., $\bm{x},\bm{A}$) to denote vectors or matrices, and use corresponding non-bold letters to denote their components. We let $\e$ be the all-one vector or matrix whenever necessary, let $\bm{0}$ be the all-zero vector or matrix whenever necessary, and we let $\e_i$ be the $i$th standard basis vector. 
Given an integer $n$, we let $[n]:=\{1,2,\ldots,n\}$, and use $\Re_+^n:=\{\bm{x}\in \Re^n:x_l\geq0, \forall l\in [n]\}$ and $\Re_-^n:=\{\bm{x}\in \Re^n:x_l\leq0, \forall l\in [n]\}$. Given a real number $t$, we let $(t)_+:=\max\{t,0\}$. Given a finite set $I$, we let $|I|$ denote its cardinality. We let $\tilde{\bm{\xi}}$ denote a random vector with support $\Xi$ and denote one of its realization by $\bm{\xi}$. Given a real-valued random variable $\tilde{\xi}:\Omega\rightarrow \Re$ with probability distribution $\Pr$, its \text{ess.sup}$(X):=\inf\{c: \Pr\{\omega:\tilde{\xi}(\omega)>c\}=1\}$. Given a set $R$, the characteristic function $\chi_{R}(\bm{x})=0$ if $\bm{x}\in R$, and $\infty$, otherwise, while the indicator function $\I(\bm{x}\in R)$ =1 if $\bm{x}\in R$, and 0, otherwise. 
We let $\bm{I}_n$ denote $n\times n$ identify matrix.
For a vector $\bm{a}$, we let $|\bm{a}|$ denote the result by taking element-wise absolute and let $(\bm{a})_+=\max\{\bm{a},0\}$ by taking element-wise maximum. 
For a matrix $\bm{A}$, we let $|\bm{A}|$ denote the result by taking element-wise absolute, let $(\bm{A})_+=\max\{\bm{A},0\}$ by taking element-wise maximum, and let $\|\bm{A}\|_p$ denote its element-wise $p$-norm. 
Additional notation will be introduced as needed. 

\section{Preliminaries}\label{sec_prelim}

Similar to \cite{hanasusanto2016conic}, we will make the following assumption throughout this paper.
\begin{itemize}
\item (Sufficiently Expensive Recourse) For any $\bm{x}\in \X$, the dual of the second-stage problem \eqref{recourse} is feasible for all $\bxi\in \Xi$. 
\end{itemize}
Note that this assumption is used to ensure that the strong duality of the second-stage problem \eqref{recourse} always holds. If this assumption does not hold, then the proposed reformulations in this paper might not be exact.

According to the strong duality of distributionally robust optimization with $\infty-$Wasserstein ambiguity set \citep{bertsimas2018data}, we observe that the function $\Z(\bm{x})$ can be equivalently represented as the following bilinear program.
\begin{lemma}\label{lem_equivalent}
the function $\Z(\bm{x})$ is equivalent to 
\begin{align}
	&\Z(\bm{x})= \frac{1}{N}\sum_{j\in [N]}\sup_{(\bxi_q,\bxi_T)\in \Xi,\bm{\pi}\in \Re^{\ell}_+}\left\{(\bm{h}(\bm{x})-\bm{T}(\bfx)\bxi_T)^\top\bm{\pi}:\|(\bxi_q,\bxi_T)-(\bfzeta_q^j,\bfzeta_T^j)\|_p\leq \theta,\bm{W}^\top \bm{\pi}= \bm{Q}\bxi_q+\bm{q}\right\}.\label{sp_svm2}
	\end{align} 
\end{lemma}
\begin{subequations}
\begin{proof} According to Theorem 5 in \cite{bertsimas2018data}, $\Z(\bm{x})=\sup_{\Pr\in \P}\E_{\Pr}[Z(\bm{x},\tilde{\bxi})]$ is equivalent to
	\begin{align}
	&\Z(\bm{x})= \frac{1}{N}\sum_{j\in [N]}\sup_{\rxi}\left\{Z(\bm{x},{\bm{\xi}}):\rxi\in \Xi, \|\rxi-\bm{\zeta}^j\|_p\leq \theta\right\}.\label{sp_svm1}
	\end{align}
Suppose $\bm{\pi}$ is the dual vector associated with constraints \eqref{sp-det_2nd}, then we can equivalently  represent $Z(\bm{x},{\bm{\xi}})$ as
	\begin{align}
Z(\bm{x},\bm{\xi})=\max_{\bm{\pi}}\left\{(\bm{h}(\bm{x})-\bm{T}(\bfx)\bxi_T)^\top\bm{\pi}:\bm{W}^\top \bm{\pi}= \bm{Q}\bxi_q+\bm{q}, \bm{\pi}\in \Re^{\ell}_+\right\}.\label{sp-obj_recourse_dual}
\end{align}
Substituting \eqref{sp-obj_recourse_dual} into \eqref{sp_svm1} and using the fact that $\bxi=(\bxi_q,\bxi_T)$ and $\bm\zeta^j=(\bfzeta_q^j,\bfzeta_T^j)$, we arrive at \eqref{sp_svm2}.  

\qedA
\end{proof}
\end{subequations}
Note that the inner supremum of \eqref{sp_svm2} is to maximize bilinear objective function over convex constraints, which is often difficult to solve. Therefore, the main focus of this paper is to study the complexity of evaluating the function $\Z(\bfx)$ and provide sufficient conditions under which the inner supremum is efficiently solvable. 

Other useful tools that this paper relies on are summarized below.
\begin{property} 
\begin{enumerate}[label=(\roman*)]
\item (Dual Norm, \citealt{rockafellar1970convex}) For any norm $\|\cdot\|_p$ with $p\in [1,\infty]$, its dual norm is $$\|\bm{r}\|_{p*}=\max_{\bm{s}}\left\{\bm{r}^\top\bm{s}:\|\bm{s}\|_p\leq 1\right\},$$
where $p*=\frac{p}{p-1}$;
\item (Integral Polyhedron, \citealt{schrijver1998theory}) Given a rational polyhedron $P=\{\bm r\in \Re^n: \bm{A}\bm r\geq \bm b\}$ is integral if and only if $P=\conv(P\cap \Ze^n)$;
\item (Tractability, \citealt{ben2009robust})  We say the function $\Z(\bfx)$ has a tractable representation, if for any given $\bm{\bfx}\in \Re^{n_1}$, there exists an efficient algorithm which can evaluate the function $\Z(\bfx)$ in time polynomial in $n_1,n_2,m_2,m_2,\ell, N$.
\end{enumerate}
\end{property}

\section{Continuous Support: Tractable Reformulations and Complexity Analysis}\label{sec_cont}

In this section, we first provide the tractable representations of the function $\Z(\bm{x})$ under various settings and then show that in general, it is NP-hard to evaluate the function $\Z(\bm{x})$. We split this section into four parts, which include tractable reformulations of general DRTSP, special DRTSP with objective uncertainty only, special DRTSP  with constraint uncertainty only, and complexity analysis.

\subsection{Tractable Reformulation I: General DRTSP with $L_\infty$ Reference Distance}
For the general DRTSP, we show that the function $\Z(\bm{x})$ has a tractable representation given that the reference distance is $\|\cdot\|_p=\|\cdot\|_\infty$ (i.e., $p=\infty$) and the image of the technology mapping $\T(\bfx)$ is always non-negative or non-positive.
\begin{theorem}\label{thm_tractable1}
Suppose that $\Xi=\Re^{m_1}\times \Re^{m_2}$. 
If $p=\infty$ and $\T(\bfx)\in \Re_+^{\ell\times m_1}$ or $\T(\bfx)\in \Re_-^{\ell\times m_1}$, then the function $\Z(\bm{x})$ is equivalent to 
\begin{align}
	&\Z(\bm{x})= \frac{1}{N}\sum_{j\in [N]}\min_{ \bm{y}\in \Re^{n_2}}	\left\{(\bm{Q}\bfzeta_q^j+\bm{q})^\top \bm{y}+\theta \|\bm{Q}^\top\bm{y}\|_1:\bm{T}(\bfx)\bfzeta^j_T+\W\bm{y}-\theta|\T(\bm x)|\e\geq \bm{h}(\bm{x}) \right\}.\label{sp_svm8}
	\end{align}


\end{theorem}
\begin{subequations}
\begin{proof}
Since $\Xi=\Re^{m_1}\times \Re^{m_2}$ and $p=\infty$, thus \eqref{sp_svm2} becomes
\begin{align*}
	&\Z(\bm{x})= \frac{1}{N}\sum_{j\in [N]}\sup_{\bm{\pi}\in \Re^{\ell}_+,\bxi_q,\bxi_T}\left\{(\bm{h}(\bm{x})-\bm{T}(\bfx)\bxi_T)^\top\bm{\pi}:\|\bxi_q-\bfzeta_q^j\|_\infty\leq \theta,\|\bxi_T-\bfzeta_T^j\|_\infty\leq \theta,\bm{W}^\top \bm{\pi}= \bm{Q}\bxi_q+\bm{q}\right\}.
	\end{align*}
Above, optimizing $\bxi_T$ and using the dual norm of $\|\cdot\|_\infty$, we have
\begin{align}
	&\Z(\bm{x})= \frac{1}{N}\sum_{j\in [N]}\sup_{ \bm{\pi}\in \Re^{\ell}_+,\bxi_q}\left\{(\bm{h}(\bm{x})-\bm{T}(\bfx)\bfzeta^j_T)^\top\bm{\pi}+\theta\|\bm{T}(\bfx)^\top \bm{\pi}\|_{1}:\|\bxi_q-\bfzeta_q^j\|_\infty\leq \theta, \bm{W}^\top \bm{\pi}= \bm{Q}\bxi_q+\bm{q} \right\}.\label{sp_svm41}
\end{align} 
Note that since $\T(\bfx)\in \Re_+^{\ell\times m_1}$ or $\T(\bfx)\in \Re_-^{\ell\times m_1}$, thus $\|\T(\bfx)^\top\bm\pi\|_{1}=\e^\top|\T(\bfx)|^\top \bm{\pi}$. 
Let $\bm{y}$ denote the dual variables of the constraints $\bm{W}^\top \bm{\pi}= \bm{Q}\bxi_q+\bm{q}$. Then according to the strong duality of linear programming, \eqref{sp_svm41} is equivalent to
\begin{align}
\Z(\bm{x})= \frac{1}{N}\sum_{j\in [N]}\min_{ \bm{y}\in \Re^{n_2}}	&\sup_{ \bm{\pi}\in \Re^{\ell}_+,\bxi_q}\left\{(\bm{h}(\bm{x})-\bm{T}(\bfx)\bfzeta^j_T)^\top\bm{\pi}+\theta\e^\top|\T(\bfx)|\top \bm{\pi}+\bm{y}^\top(\bm{Q}\bxi_q+\bm{q}-\bm{W}^\top \bm{\pi}):\right.\notag\\
&\left.\|\bxi_q-\bfzeta_q^j\|_\infty\leq \theta \right\},\label{sp_svm52}
\end{align} 
which is equivalent to \eqref{sp_svm8} by optimizing over $(\bxi_q,\bm\pi)$.
  \qedA
\end{proof}
\end{subequations}

We make the following remarks about \Cref{thm_tractable1} and its corresponding formulation \eqref{sp_svm8}.
\begin{enumerate}[label=(\roman*)]
\item We can introduce auxiliary variables to linearize the terms $ \|\bm{Q}^\top\bm{y}\|_1,|\T(\bm x)|$ and reformulate the minimization problem \eqref{sp_svm8} as a linear program;
\item If $\theta=0$, i.e., if the empirical distribution is sufficient to describe the probability of random parameters, then
\begin{align}
\Z(\bm{x})=\frac{1}{N}\sum_{j\in [N]}Z(\bm{x},\bm{\zeta}^j);
\end{align}
\item The extra terms, $\theta \|\bm{Q}^\top\bm{y}\|_1$ in the objective and $-\theta|\T(\bm x)|_{1}\e$ in the constraints, enforce the robustness of the proposed formulation due to ambiguous distributional information. These terms will vanish if more and more observations have been made to drive the Wasserstein radius to be 0. For more discussions about asymptotic behavior of Wasserstein ambiguity sets, interested readers are referred to \cite{bertsimas2019twostage,bertsimas2018data,blanchet2016quantifying,esfahani2015data,hanasusanto2016conic,xie2018drccp}; 
\item If the assumption that $\T(\bfx)\in \Re_+^{\ell\times m_1}$ or $\T(\bfx)\in \Re_-^{\ell\times m_1}$ does not hold, then \eqref{sp_svm8} provides an upper bound for $\Z(\bfx)$ and this upper bound will become exact when $\theta\rightarrow 0$; and
\item Similarly, if the reference distance is defined by other norm $\|\cdot\|_p$, then according to the following formula
\[\|\bxi\|_p\leq \sqrt[p]{m_1+m_2}\|\bxi\|_\infty.\]
Thus, \eqref{sp_svm8} provides an upper bound for $\Z(\bfx)$ by inflating $\theta$ to $\sqrt[p]{m_1+m_2}\theta$ and this upper bound will become exact when $\theta\rightarrow 0$.
\end{enumerate}

According to the representation result in \Cref{thm_tractable1}, we provide the following equivalent deterministic reformulation of DRTSP \eqref{sp}. 
\begin{proposition}\label{thm_tractable1_deterministic}
Suppose that $\Xi=\Re^{m_1}\times \Re^{m_2}$. 
If $p=\infty$ and $\T(\bfx)\in \Re_+^{\ell\times m_1}$ or $\T(\bfx)\in \Re_-^{\ell\times m_1}$, then DRTSP \eqref{sp} is equivalent to 
\begin{subequations}
\begin{align}
v^*= \min_{\bm{x},\bm{y}}\quad& \bm{c}^{\top}\bm{x}+\frac{1}{N}\sum_{j\in [N]}[(\bm{Q}\bfzeta_q^j+\bm{q})^\top \bm{y}^j+\theta \|\bm{Q}^\top\bm{y}^j\|_1],\\
\textrm{s.t.   }	\quad&\bm{T}(\bfx)\bfzeta^j_T+\W\bm{y}^j-\theta|\T(\bm x)|_{1}\e\geq \bm{h}(\bm{x}),\forall j\in [N],\\
&\bm{x}\in \X,\bm{y}^j\in \Re^{n_2},\forall j\in [N].
	\end{align}
\end{subequations}

\end{proposition}

The following example illustrates how to use the proposed formulation in practical application problems.

\begin{example}\label{example1}\rm \textbf{(Reliable Facility Location Problem (RFLP) under Probabilistic Disruptions)} Let us consider a two-stage facility location problem with random demands and probabilistic disruptions, an extension of the work \citep{Cui10,lu2015reliable}. Suppose a warehousing company needs to build facilities at candidate locations indexed by $[n_1]$, which are required to serve customers at locations indexed by $[\ell]$. Each facility $s\in [n_1]$ bears a setup cost $f_s$ and due to catastrophic events (e.g., hurricane, power outage, etc.), it might be disrupted, thus, we use $\tilde{\delta}_s\in \{0,1\}$ to denote its status, i.e., $\tilde{\delta}_s=1$ if it will function well, 0, otherwise. We suppose that each customer $t\in [\ell]$ has a stochastic demand $\tilde{d}_t$ and incurs a unit transportation cost for a shipment from facility $s\in [n_1]$, denoted by $c_{ts}$. The random parameters $\tilde{\bxi}=(\tilde{\bm\delta}, \tilde{\bm d})$. Suppose there are $N$ empirical data points available, denoted by $\{\bm\zeta^j:=(\hat{\bm{\delta}}^j,\hat{\bm{d}}^j)\}_{j\in [N]}$.

To ensure the feasibility of the model, similar to \cite{Cui10,lu2015reliable}, we assume that there is an emergency (or dummy) facility indexed by $n_1+1$, which will be never disrupted, and its unit transportation cost for each customer $t\in [\ell]$ is $c_{t(n_1+1)}=M$, where $M$ is a large number. Under this setting, distributionally robust RFLP (DR-RFLP) can be formulated as
\begin{subequations}\label{DR-FRLP}
\begin{align}
v^*= \min_{\bm{x}} & \ \ \bm{f}^{\top}\bm{x}+\Z(\bm{x}), \label{RFLP_sp-obj} \\
\rm{s.t. } 
& \ \ \bm{x} \in \{0,1\}^{n_1}, \label{RFLP_sp-det}\\
& \ \ \Z(\bm{x})=\sup_{\Pr\in \P}\E_{\Pr}[Z(\bm{x},\tilde{\bm{\xi}})],\label{RFLP_sp-recourse}
\end{align}
\end{subequations}
where the recourse function is 
\begin{subequations}\label{RFLP_recourse}
\begin{align}
Z(\bm{x},\bm{\xi})=\min_{\bm{y}} & \ \ \sum_{t\in [\ell]}\sum_{s\in [n_1+1]}c_{ts}d_t y_{ts}, \label{RFLP_sp-obj_2nd} \\
\rm{s.t. } & \ \ \sum_{s\in [n_1+1]}y_{ts}=1, \forall t\in [\ell], \label{RFLP_sp-det_2nd1}\\
&\ \ y_{ts} \leq \delta_s x_s, \forall t\in [\ell], \forall s\in [n_1],\label{RFLP_sp-det_2nd2}\\
& \ \ \bm{y}\in \Re_+^{\ell\times n_1}.\label{RFLP_sp-bound_2nd}
\end{align}
\end{subequations}

Suppose the reference distance is $\|\cdot\|_{\infty}$ and the support of $\tilde{\bm{\xi}}$ is $\Re^{n_1}\times \Re^\ell$. Since the coefficients of uncertain parameters $\tilde{\bm{\delta}}$ in the constraints \eqref{RFLP_sp-det_2nd2} always have the same sign, according to \Cref{thm_tractable1_deterministic}, DR-RFLP can be equivalently formulated as the following mixed integer linear program (MILP):
\begin{subequations}\label{DR-FRLP1}
\begin{align}
v^*= \min_{\bm{x},\bm{y}} & \ \ \bm{f}^{\top}\bm{x}+\frac{1}{N}\sum_{j\in [N]}\sum_{t\in [\ell]}\sum_{s\in [n_1+1]}c_{ts}(\hat{d}_t^j +\theta)y_{ts}^j, \label{RFLP_sp-obj1} \\
\rm{s.t. } 
& \ \ \sum_{s\in [n_1+1]}y_{ts}^j=1, \forall j\in [N],\forall t\in [\ell],\\
&\ \ y_{ts}^j \leq (\hat{\delta}_s^j-\theta) x_s, \forall j\in [N],\forall t\in [\ell],\forall s\in [n_1], \label{RFLP_sp-det_2nd21}\\
& \ \ \bm{x} \in \{0,1\}^{n_1}, \bm{y}^j\in \Re_+^{\ell\times n_1},\forall j\in [N]. \label{RFLP_sp-det1}
\end{align}
\end{subequations}
\qedA
\end{example}

\subsection{Tractable Reformulation II: With Objective Uncertainty Only}
If there are only objective uncertainty involved in DRTSP, then the function $\Z(\bm{x})$ always has a tractable representation provided that the reference distance is $\|\cdot\|_p$ for any $p\in [1,\infty]$.

\begin{theorem}\label{thm_tractable1_AC}
Suppose that $\Xi=\Re^{m_1}\times\{\bxi_T\}$. Then for any $p\in [1,\infty]$, the function $\Z(\bm{x})$ is equivalent to 
\begin{align}
	&\Z(\bm{x})= \frac{1}{N}\sum_{j\in [N]}\min_{ \bm{y}\in \Re^{n_2}}	\left\{(\bm{Q}\bfzeta_q^j+\bm{q})^\top \bm{y}+\theta \|\bm{Q}^\top\bm{y}\|_{p*}:\bm{T}(\bfx)\bxi_T+\W\bm{y}\geq \bm{h}(\bm{x}) \right\},\label{sp_svm8_AC}
	\end{align}
where $\|\cdot\|_{p*}$ denotes the dual norm of $\|\cdot\|_{p}$ with $p*=\frac{p}{p-1}$.
\end{theorem}
\begin{subequations}

\begin{proof}
Since $\Xi=\Re^{m_1}\times\{\bxi_T\}$, \eqref{sp_svm2} becomes
\begin{align}
	&\Z(\bm{x})= \frac{1}{N}\sum_{j\in [N]}\sup_{ \bm{\pi}\in \Re^{\ell}_+,\bxi_q}\left\{(\bm{h}(\bm{x})-\bm{T}(\bfx)\bxi_T)^\top\bm{\pi}:\|\bxi_q-\bfzeta_q^j\|_p\leq \theta, \bm{W}^\top \bm{\pi}= \bm{Q}\bxi_q+\bm{q} \right\}.\label{sp_svm41_AC}
\end{align} 
Let $\bm{y}$ denote the dual variables of the constraints $\bm{W}^\top \bm{\pi}= \bm{Q}\bxi_q+\bm{q}$. Since the inner supremum of \eqref{sp_svm41_AC} is essentially strictly feasible,  according to the strong duality of conic programming \citep{ben2001lectures}, \eqref{sp_svm41} is equivalent to
\begin{align}
\Z(\bm{x})= \frac{1}{N}\sum_{j\in [N]}\min_{ \bm{y}\in \Re^{n_2}}	&\sup_{ \bm{\pi}\in \Re^{\ell}_+,\bxi_q}\left\{(\bm{h}(\bm{x})-\bm{T}(\bfx)\bxi_T)^\top\bm{\pi}+\bm{y}^\top(\bm{Q}\bxi_q+\bm{q}-\bm{W}^\top \bm{\pi})\right.\notag\\
&\left.\|\bxi_q-\bfzeta_q^j\|_p\leq \theta \right\},\label{sp_svm52_AC}
\end{align} 
which is further equivalent to \eqref{sp_svm8} by optimizing over $(\bm{\pi},\bxi_q)$.
  \qedA
\end{proof}
\end{subequations}
We make the following remarks about \Cref{thm_tractable1_AC} and its corresponding formulation \eqref{sp_svm8_AC}.
\begin{enumerate}[label=(\roman*)]
\item For any rational $p\in [1,\infty]$, the penalty term $\theta \|\bm{Q}^\top\bm{y}\|_{p*}$ is second order conic representable \citep{ben2001lectures}. Therefore, \eqref{sp_svm8_AC} can be further reformulated as a second order conic program; and
\item The penalty term, $\theta \|\bm{Q}^\top\bm{y}\|_{p*}$ in the objective, enforces the robustness of the proposed model due to ambiguous distributional information. This term will vanish if more and more observations have been made to drive the Wasserstein radius to 0.
\end{enumerate}

We provide the following equivalent deterministic reformulation of DRTSP \eqref{sp} with objective uncertainty only. 
\begin{proposition}\label{thm_tractable1_deterministic_AC}
Suppose that $\Xi=\Re^{m_1}\times\{\bxi_T\}$. Then for any $p\in [1,\infty]$, DRTSP \eqref{sp} is equivalent to 
\begin{subequations}
\begin{align}
v^*= \min_{\bm{x},\bm{y}}\quad& \bm{c}^{\top}\bm{x}+\frac{1}{N}\sum_{j\in [N]}\left[(\bm{Q}\bfzeta_q^j+\bm{q})^\top \bm{y}^j+\theta \|\bm{Q}^\top\bm{y}^j\|_{p*}\right],\\
\textrm{s.t.   }	\quad&\bm{T}(\bfx)\bxi_T+\W\bm{y}^j\geq \bm{h}(\bm{x}),\forall j\in [N],\\
&\bm{x}\in \X,\bm{y}^j\in \Re^{n_2},\forall j\in [N].
	\end{align}
\end{subequations}

\end{proposition}

We will illustrate the proposed formulation using \Cref{example1}, where we suppose that there are no disruption risks, i.e., the only uncertain parameters are customers' demands.
\begin{example}\label{example2}\rm Following the notation in \Cref{example1}, let us consider DR-RFLP with demand uncertainty only, i.e., the random parameters $\tilde{\bm\delta}$ satisfy $\Pr\{\tilde{\bm\delta}=\bm \delta\}=1$.

Suppose the reference distance is $\|\cdot\|_{p}$ and the support of $\tilde{\bm{\xi}}$ is $\{\bm\delta\}\times \Re^\ell$. According to \Cref{thm_tractable1_deterministic_AC}, DR-FRLP with demand uncertainty only can be equivalently formulated as the following mixed integer conic program (MICP):
\begin{subequations}\label{DR-FRLP12}
\begin{align}
v^*= \min_{\bm{x},\bm{y}} & \ \ \bm{f}^{\top}\bm{x}+\frac{1}{N}\sum_{j\in [N]}\left[\sum_{t\in [\ell]}\sum_{s\in [n_1+1]}c_{ts}\hat{d}_t^j y_{ts}^j+\theta \sqrt[p*]{\sum_{t\in [\ell]}\left(\sum_{s\in [n_1+1]}c_{ts}y_{ts}^j\right)^{p*}}\right], \label{RFLP_sp-obj12} \\
\rm{s.t. } 
& \ \ \sum_{s\in [n_1+1]}y_{ts}^j=1, \forall j\in [N],\forall t\in [\ell],\\
&\ \ y_{ts}^j \leq \delta_sx_s, \forall j\in [N],\forall t\in [\ell],\forall s\in [n_1], \label{RFLP_sp-det_2nd212}\\
& \ \ \bm{x} \in \{0,1\}^{n_1}, \bm{y}^j\in \Re_+^{\ell\times n_1},\forall j\in [N]. \label{RFLP_sp-det12}
\end{align}
\end{subequations}
\QEDA
\end{example}

\subsection{Tractable Reformulation III: With Constraint Uncertainty Only}

If there are only constraint uncertainty involved in DRTSP, then the function $\Z(\bm{x})$ can have a tractable representation given that the reference distance when $p=1$.
\begin{theorem}\label{thm_approximation1} 
Suppose that $\Xi= \{\bxi_q\}\times \Re^{m_2}$ and $p=1$. Then the function $\Z(\bm{x})$ is equivalent to 
\begin{align}
	&\Z(\bm{x})= \frac{1}{N}\sum_{j\in [N]}\max_{r\in \{-1,1\}}\max_{i\in [m_1]}\min_{ \bm{y}\in \Re^{n_2}}	\left\{(\bm{Q}\bxi_q+\bm{q})^\top \bm{y}:\bm{T}(\bfx)\bfzeta^j_T+\W\bm{y}-\theta r\T(\bm x)\e_i\geq \bm{h}(\bm{x}) \right\}.\label{sp_svm5}
	\end{align}
\end{theorem}
\begin{subequations}
\begin{proof}
Since $\Xi=\{\bxi_q\}\times \Re^{m_2}$ and $p=1$, \eqref{sp_svm2} becomes
\begin{align}
	&\Z(\bm{x})= \frac{1}{N}\sum_{j\in [N]}\sup_{ \bm{\pi}\in \Re^{\ell}_+,\bxi_T}\left\{(\bm{h}(\bm{x})-\bm{T}(\bfx)\bxi_T)^\top\bm{\pi}:\|\bxi_T-\bfzeta_T^j\|_1\leq \theta, \bm{W}^\top \bm{\pi}= \bm{Q}\bxi_q+\bm{q} \right\},\label{sp_svm3_AO}
	\end{align} 
Above, optimizing $\bxi_T$ involving dual norm of $\|\cdot\|_1$, we have
\begin{align}
	&\Z(\bm{x})=\frac{1}{N}\sum_{j\in [N]}\sup_{ \bm{\pi}\in \Re^{\ell}_+}\left\{(\bm{h}(\bm{x})-\bm{T}(\bfx)\bfzeta_T^j)^\top\bm{\pi}+\theta\|\bm{T}(\bfx)^\top\bm\pi\|_{\infty}:\bm{W}^\top \bm{\pi}= \bm{Q}\bxi_q+\bm{q} \right\}.\label{sp_svm4_AO}
	\end{align}
Since
\[\|\bm{T}(\bfx)^\top\bm\pi\|_{\infty}=\max_{i\in [m_1]}\max\{(\bm{T}(\bfx)^\top\bm\pi)_i,-(\bm{T}(\bfx)^\top\bm\pi)_i\}\]
thus, \eqref{sp_svm4_AO} is further equivalent to
\begin{align}
	&\Z(\bm{x})=\frac{1}{N}\sum_{j\in [N]}\max_{r\in \{-1,1\}}\max_{i\in [m_1]}\sup_{ \bm{\pi}\in \Re^{\ell}_+}\left\{(\bm{h}(\bm{x})-\bm{T}(\bfx)\bfzeta_T^j)^\top\bm{\pi}+\theta r\e_i^\top\bm{T}(\bfx)^\top\bm\pi:\bm{W}^\top \bm{\pi}= \bm{Q}\bxi_q+\bm{q} \right\},\label{sp_svm4_AO1}
	\end{align} 
Taking the dual of inner supremum and using strong duality of linear programming, we arrive at \eqref{sp_svm5}.
  \qedA
\end{proof}
\end{subequations}

We make the following remarks about \Cref{thm_approximation1} and its corresponding formulation \eqref{sp_svm5}.
\begin{enumerate}[label=(\roman*)]
\item Clearly, since DRTSP with constraint uncertainty only is a special case of general DRTSP, thus the result from \Cref{thm_tractable1} directly follows and is not listed here;
\item \cite{hanasusanto2016conic} also proved that under the setting of \Cref{thm_approximation1}, DRTSP with $1$-Wasserstein ambiguity set is tractable. However, our formulation and required proof technique are quite different from theirs;
\item To obtain $\Z(\bfx)$, one needs to solve $2m_1$ linear programs for each $j\in [N]$;
\item If $\T(\bfx)\in \Re_+^{\ell\times m_1}$ or $\T(\bfx)\in \Re_-^{\ell\times m_1}$, then due to monotonicity, we must have optimal $r^*=1$ or $r^*=-1$, respectively. Thus, for these cases, one only needs to solve $m_1$ linear programs instead of $2m_1$ for each $j\in [N]$; and
\item The penalty term, $-\theta r\T(\bm x)\e_i$ in the constraints, enforces the robustness of the proposed model due to ambiguous distributional information. 
\end{enumerate}

In view of the result in \Cref{thm_approximation1}, we provide the following equivalent deterministic reformulation of DRTSP \eqref{sp}.
\begin{proposition}\label{thm_approximation1_deterministic}
Suppose that $\Xi= \{\bxi_q\}\times \Re^{m_2}$ and $p=1$. Then DRTSP \eqref{sp} is equivalent to 
\begin{subequations}
\begin{align}
v^*= \min_{\bm{x},\bm{\eta}}\quad& \bm{c}^{\top}\bm{x}+\frac{1}{N}\sum_{j\in [N]}\eta_j,\\
\textrm{s.t.   }	\quad&\eta_{j}\geq (\bm{Q}\bxi_q^j+\bm{q})^\top \bm{y}^{ijr},\forall j\in [N],\forall i\in [m_1],\forall r\in \{-1,1\},\\
&\bm{T}(\bfx)\bfzeta_T^j+\W\bm{y}^{ijr}-\theta r\T(\bm x)\e_i \geq \bm{h}(\bm{x}),\forall j\in [N],\forall i\in [m_1],\forall r\in \{-1,1\},\\
&\bm{x}\in \X,\bm{y}^{ijr}\in \Re^{n_2} ,\forall j\in [N],\forall i\in [m_1],\forall r\in \{-1,1\}.
	\end{align}
\end{subequations}
\end{proposition}

Another special case of DRTSP without objective uncertainty is that the dual constraint system of \eqref{recourse} is bounded and has a small number of extreme points. In this case, equivalently, we can represent the recourse function in the form of piece-wise max of a finite number of affine functions in the random parameters, and obtain the tractable reformulation for any reference distance $\|\cdot\|_p$ for any $p\in [1,\infty]$. This result is summarized below.
\begin{proposition}\label{prop_max}Suppose that $\Xi=\Re^{\tau}$ and $z(\bm{x})=\sup_{\Pr\in \P}\E_{\Pr}[\max_{i\in [m]}\{\bm{a}_i(\bm{x})^\top \bm{\xi}+d_i(\bm{x})\}]$ with affine functions $\bm{a}_i(\bm{x}):\Re^{n_1}\rightarrow\Re^\tau$ and $d_i(\bm{x}):\Re^{n_1}\rightarrow\Re$ for each $i\in [m]$. Then 
\begin{itemize}
\item  Function $z(\bm{x})$ is equivalent to
\begin{align}\label{eq_rhs_uncertainty}
z(\bm{x})=\frac{1}{N}\sum_{j\in [N]}\max_{i\in [m]}\left[\bm{a}_i(\bm{x})^\top \hat{\bm{\zeta}}^j+d_i+\theta\|\bm{a}_i(\bm{x})\|_{p*}\right].
\end{align}
\item DRTSP \eqref{sp} is equivalent to 
\begin{subequations}\label{eq_reform_drtsp}
\begin{align}
v^*= \min_{\bm{x},\bm{\eta}}\quad& \bm{c}^{\top}\bm{x}+\frac{1}{N}\sum_{j\in [N]}\eta_j,\\
\textrm{s.t.   }	\quad&\eta_j\geq \bm{a}_i(\bm{x})^\top \hat{\bm{\zeta}}^j+d_i+\theta\|\bm{a}_i(\bm{x})\|_{p*} ,\forall j\in [N],\forall i\in [m],\\
&\bm{x}\in \X.
	\end{align}
\end{subequations}
\end{itemize}

\end{proposition}
\begin{proof}
Since $\Xi=\Re^{\tau}$ and $Z(\bfx,\bxi)=\max_{i\in [m]}\{\bm{a}_i(\bm{x})^\top \bm{\xi}+d_i(\bm{x})\}$, \eqref{sp_svm1} becomes
\begin{align*}
	&\Z(\bm{x})= \frac{1}{N}\sum_{j\in [N]}\max_{i\in [m]}\sup_{ \bxi}\left\{\bm{a}_i(\bm{x})^\top \bm{\xi}+d_i(\bm{x}):\|\bxi-\bfzeta^j\|_p\leq \theta\right\},
	\end{align*} 
Above, optimizing $\bxi$ using dual norm of $\|\cdot\|_p$, we arrive at \eqref{eq_rhs_uncertainty}.

The formulation \eqref{eq_reform_drtsp} follows from a straightforward linearization.
  \qedA
\end{proof}

We will illustrate the proposed formulation in \Cref{thm_approximation1_deterministic} using \Cref{example1}, where we assume that there is no demand uncertainty, i.e., the only uncertain parameters are facility disruptions.
\begin{example}\label{example3}\rm Following the notation in \Cref{example1}, let us consider DR-RFLP with disruption risks only, i.e., the random parameters $\tilde{\bm d}$ satisfy $\Pr\{\tilde{\bm d}=\bm d\}=1$.

Suppose the reference distance is $\|\cdot\|_{1}$ and the support of $\tilde{\bm{\xi}}$ is $\Re^{n_1}\times \{\bm d\}$. 
According to \Cref{thm_approximation1_deterministic}, DR-FRLP with disruption risks can be equivalently formulated as the following MILP:
\begin{subequations}\label{DR-FRLP13}
\begin{align}
v^*= \min_{\bm{x},\bm{y}} & \ \ \bm{f}^{\top}\bm{x}+\frac{1}{N}\sum_{j\in [N]}\eta_j, \label{RFLP_sp-obj13} \\
\rm{s.t. } 
&\ \  \eta_j\geq \sum_{t\in [\ell]}\sum_{s\in [n_1+1]}c_{ts}d_t y_{ts}^{ijr}, \forall j\in [N],\forall i\in [n_1],\forall r\in \{-1,1\},\\
& \ \ \sum_{s\in [n_1+1]}y_{ts}^{ijr}=1, \forall j\in [N],\forall t\in [\ell],\forall i\in [n_1],\forall r\in \{-1,1\},\\
&\ \ y_{ts}^{ijr} \leq \hat{\delta}_s^jx_s-\I(s=i)\theta r x_s, \forall j\in [N],\forall t\in [\ell],\forall i\in [n_1],\forall s\in [n_1],\forall r\in \{-1,1\}, \label{RFLP_sp-det_2nd213}\\
& \ \ \bm{x} \in \{0,1\}^{n_1}, \bm{y}^{ijr}\in \Re_+^{\ell\times n_1},\forall j\in [N],\forall i\in [n_1],\forall r\in \{-1,1\}. \label{RFLP_sp-det13}
\end{align}
\end{subequations}
Note that due to monotonicity, in the above formulation, the optimal $r^*=1$. Thus, DR-FRLP with disruption risks can be further simplified as
\begin{subequations}\label{DR-FRLP132}
\begin{align}
v^*= \min_{\bm{x},\bm{y}} & \ \ \bm{f}^{\top}\bm{x}+\frac{1}{N}\sum_{j\in [N]}\eta_j, \label{RFLP_sp-obj132} \\
\rm{s.t. } 
&\ \  \eta_j\geq \sum_{t\in [\ell]}\sum_{s\in [n_1+1]}c_{ts}d_t y_{ts}^{ij}, \forall j\in [N],\forall i\in [n_1],\\
& \ \ \sum_{s\in [n_1+1]}y_{ts}^{ij}=1, \forall j\in [N],\forall t\in [\ell],\forall i\in [n_1],\\
&\ \ y_{ts}^{ij} \leq \hat{\delta}_s^jx_s-\I(s=i)\theta x_s, \forall j\in [N],\forall t\in [\ell],\forall i\in [n_1],\forall s\in [n_1], \label{RFLP_sp-det_2nd2132}\\
& \ \ \bm{x} \in \{0,1\}^{n_1}, \bm{y}^{ij}\in \Re_+^{\ell\times n_1},\forall j\in [N],\forall i\in [n_1]. \label{RFLP_sp-det132}
\end{align}
\end{subequations}
\QEDA
\end{example}

\subsection{Complexity Analysis} 

We close this section by showing that for general reference distance $\|\cdot\|_{p}$ with $p\in(1,\infty]$, computing the function $\Z(\bm{x})$ with $N=1$ is NP-hard.
\begin{proposition}\label{prop_complexity}
Computing $\Z(\bm{x})$ is NP-hard whenever 
the reference distance is $\|\cdot\|_{p}$ with any $p\in(1,\infty]$, $N=1$, $\Xi= \{\bxi_q\}\times \Re^{m_2}$, $\bm{h}(\bm{x})=\bm{0},\bfzeta_T^1=\bm{0}$, and Wasserstein radius $\theta>0$. 

\end{proposition}
\begin{proof}Let us first consider the NP-complete problem - feasibility problem of a general binary program which asks
\begin{quote}
	{(Feasibility problem of a general binary program)} Given a matrix $\bm{A}\in \Qe^{t_1\times t_2}$ and a vector $\bm{b}\in \Qe^{t_1}$, is there exists a binary vector $\bm{r}\in \{0,1\}^{t_2}$ such that $\bm{A}\bm{r}=\bm{b}$?
\end{quote}

In the representation \eqref{sp_svm2} of the function $\Z(\bm{x})$, let $\ell=2t_2,n_2=m_2=t_1+t_2$, and $\T(\bm{x})=\begin{bmatrix}
\bm{I}_{t_2}\\
-\bm{I}_{t_2}
\end{bmatrix},\bm{W}^\top=\begin{bmatrix}
\bm{A}&\bm{0}\\
\bm{I}_{t_2}&\bm{I}_{t_2}
\end{bmatrix},  \bm{Q}=\bm{0},\q=\begin{pmatrix}
\bm{b}\\
\e
\end{pmatrix},\bm{\pi}=\begin{pmatrix}
\bm{r}\\
\bm{s}
\end{pmatrix}$. Since $N=1$, $\Xi= \{\bxi_q\}\times \Re^{m_2},\bfzeta_T^1=\bm{0}$, and $\theta>0$, according to the proof of \Cref{thm_approximation1}, $\Z(\bm{x})$ becomes
\begin{align}
	&\Z(\bm{x})=\sup_{ \bm{r}\in \Re^{t_2}_+,\bm{s}\in \Re^{t_2}_+}\left\{\theta\left\|\begin{pmatrix}
\bm{r}-
\bm{s}
\end{pmatrix}\right\|_{p*}:\bm{A}\bm{r}=\bm{b},\bm{r}+\bm{s}=\e \right\}.\label{sp_svm4_c}
	\end{align}
Since $p\in (1,\infty]$ and $p*=\frac{p}{p-1}\in [1,\infty)$, thus clearly, $\Z(\bm{x})=\theta\sqrt[p*]{t_2}$ if and only if there exists a binary feasible solution $(\bm{r},\bm{s})\in \{0,1\}^{t_2}\times \{0,1\}^{t_2}$ such that $\bm{A}\bm{r}=\bm{b},\bm{r}+\bm{s}=\e $, i.e., the binary program $\{\bm{r}\in \{0,1\}^{t_2}:\bm{A}\bm{r}=\bm{b}\}$ is feasible.

\qedA
\end{proof}
This result suggests that unless exploring special problem structures, the tractable results obtained in this section are sharp.

\section{Binary Support: Tractable Reformulations and Complexity Analysis}\label{sec_bin}

In practice, some stochastic programming applications might involve binary random parameters. For instance, \Cref{example1}, the disruption parameters are in fact binary, i.e., $\Pr\{\tilde{\bm{\delta}}\in \{0,1\}^{n_1}\}=1$; in the stochastic power systems with contingencies \citep{wang2012stochastic,wu2014chance}, the availability of a system component is also binary supported. Motivated by these applications, in this section, we explore the tractable representations of the function $\Z(\bm{x})$ when one of random parameters $\tilde{\bxi}_q,\tilde{\bxi}_T$ is binary, i.e., we consider either $\tilde{\bxi}_q\in \{0,1\}^{m_1}$ or $\tilde{\bxi}_T\in \{0,1\}^{m_2}$, and the other random parameters are continuous.

\subsection{Tractable Reformulation I: General DRTSP with $L_\infty$ Reference Distance}

For the general DRTSP with objective uncertainty, the function $\Z(\bm{x})$ has a tractable representation given that the reference distance is $\|\cdot\|_p=\|\cdot\|_\infty$ (i.e., $p=\infty$).
\begin{theorem}\label{D_thm_tractable1}
Suppose $p=\infty$ and $\T(\bfx)\in \Re_+^{\ell\times m_1}$ or $\T(\bfx)\in \Re_-^{\ell\times m_1}$.
\begin{enumerate}[label=(\roman*)]
\item If $\Xi=\Re^{m_1}\times \{0,1\}^{m_2}$, then the function $\Z(\bm{x})$ is equivalent to 
\begin{align}
	&\Z(\bm{x})=\begin{cases}
\displaystyle\frac{1}{N}\displaystyle\sum_{j\in [N]}\displaystyle\min_{ \bm{y}\in \Re^{n_2}}	\left\{(\bm{Q}\bfzeta_q^j+\bm{q})^\top \bm{y}+\theta \|\bm{Q}^\top\bm{y}\|_1:-(-\bm{T}(\bfx))_+\e+\W\bm{y}\geq \bm{h}(\bm{x}) \right\},&\text{ if } \theta\geq 1\\
\displaystyle\frac{1}{N}\displaystyle\sum_{j\in [N]}\displaystyle\min_{ \bm{y}\in \Re^{n_2}}	\left\{(\bm{Q}\bfzeta_q^j+\bm{q})^\top \bm{y}+\theta \|\bm{Q}^\top\bm{y}\|_1:\bm{T}(\bfx)\bfzeta^j_T+\W\bm{y}\geq \bm{h}(\bm{x}) \right\},&\text{ if } \theta<1
\end{cases}; \label{D_sp_svm81}
	\end{align}
\item If $\Xi=\{0,1\}^{m_1}\times \Re^{m_2}$ and
the polyhedron 
$\left\{(\bm\pi,\bm\xi_q)\in \Re_+^\ell\times [0,1]^{m_1}: \bm{W}^\top \bm{\pi}= \bm{Q}\bxi_q+\bm{q}\right\}$
is integral, then the function $\Z(\bm{x})$ is equivalent to 
\begin{align}
	&\Z(\bm{x})=\begin{cases}
\displaystyle\frac{1}{N}\displaystyle\sum_{j\in [N]}\displaystyle\min_{ \bm{y}\in \Re^{n_2}}	\left\{(\bm{Q}\bfzeta_q^j+\bm{q})^\top \bm{y}:\bm{T}(\bfx)\bfzeta^j_T+\W\bm{y}-\theta|\bm{T}(\bfx)|\e\geq \bm{h}(\bm{x}) \right\},&\text{ if } \theta\geq 1\\
\displaystyle\frac{1}{N}\displaystyle\sum_{j\in [N]}\displaystyle\min_{ \bm{y}\in \Re^{n_2}}	\left\{
(\bm{Q}\bfzeta_q^j+\bm{q})^\top \bm{y}+\e^\top(\bm{Q}^\top\bfy)_+:\bm{T}(\bfx)\bfzeta^j_T+\W\bm{y}-\theta|\bm{T}(\bfx)|\e\geq \bm{h}(\bm{x})
\right\},&\text{ if } \theta<1
\end{cases}. \label{D_sp_svm812}
	\end{align}

\end{enumerate}
\end{theorem}
\begin{subequations}
\begin{proof} 
We will split the proof into two parts.
\begin{enumerate}[label=(\roman*)]
\item Since $p=\infty$ and $\Xi=\Re^{m_1}\times \{0,1\}^{m_2}$, thus \eqref{sp_svm2} becomes
\begin{align}
	\Z(\bm{x})= \frac{1}{N}\sum_{j\in [N]}\sup_{ \bm{\pi}\in \Re^{\ell}_+}&\sup_{\bxi_T\in \{0,1\}^{m_2},\bxi_q}\left\{(\bm{h}(\bm{x})-\bm{T}(\bfx)\bxi_T)^\top\bm{\pi}: \right.\notag\\
&\left.\|\bxi_q-\bfzeta_q^j\|_\infty\leq \theta,\|\bxi_T-\bfzeta_T^j\|_\infty\leq \theta, \bm{W}^\top \bm{\pi}= \bm{Q}\bxi_q+\bm{q}  \right\}.\label{D_sp_svm410}
\end{align} 
Above, $\bxi_T,\bfzeta_T^j\in\{0,1\}^{m_2}$ and $\|\bxi_T-\bfzeta_T^j\|_\infty\leq \theta $ imply that if $\theta\geq 1$, then $\bxi_T\in \{0,1\}^{m_2}$; otherwise, $\bxi_T=\bfzeta_T^j$. Hence, using the assumption that $\T(\bfx)\in \Re_+^{\ell\times m_1}$ or $\T(\bfx)\in \Re_-^{\ell\times m_1}$, \eqref{D_sp_svm410} further reduces to
\begin{align}
	\Z(\bm{x})= \frac{1}{N}\sum_{j\in [N]}\sup_{ \bm{\pi}\in \Re^{\ell}_+,\bxi_q}&\left\{(\bm{h}(\bm{x})-\I(\theta< 1)\bm{T}(\bfx)\bfzeta_T^j+\I(\theta\geq 1)(-\bm{T}(\bfx))_+\e)^\top\bm{\pi}: \right.\notag\\
&\left.\|\bxi_q-\bfzeta_q^j\|_\infty\leq \theta,\bm{W}^\top \bm{\pi}= \bm{Q}\bxi_q+\bm{q}  \right\}.\label{D_sp_svm41}
\end{align} 
Following the similar linearization and dualization steps in \Cref{D_thm_tractable1}, we arrive at \eqref{D_sp_svm81}.

\item Since $p=\infty$ and $\Xi= \{0,1\}^{m_1}\times\Re^{m_2}$, thus \eqref{sp_svm2} becomes
\begin{align}
	\Z(\bm{x})= \frac{1}{N}\sum_{j\in [N]}\sup_{ \bm{\pi}\in \Re^{\ell}_+}&\sup_{\bxi_q\in \{0,1\}^{m_1},\bxi_T}\left\{(\bm{h}(\bm{x})-\bm{T}(\bfx)\bxi_T)^\top\bm{\pi}: \right.\notag\\
&\left.\|\bxi_q-\bfzeta_q^j\|_\infty\leq \theta,\|\bxi_T-\bfzeta_T^j\|_\infty\leq \theta, \bm{W}^\top \bm{\pi}= \bm{Q}\bxi_q+\bm{q}  \right\}.\label{D_sp_svm411}
\end{align} 
Optimizing over $\bxi_T$ and using the assumption that $\T(\bfx)\in \Re_+^{\ell\times m_1}$ or $\T(\bfx)\in \Re_-^{\ell\times m_1}$, \eqref{D_sp_svm411} is now equivalent to
\begin{align}
	\Z(\bm{x})= \frac{1}{N}\sum_{j\in [N]}\sup_{ \bm{\pi}\in \Re^{\ell}_+}&\sup_{\bxi_q\in \{0,1\}^{m_1}}\left\{(\bm{h}(\bm{x})-\bm{T}(\bfx)\bfzeta^j_T)^\top\bm{\pi}+\theta\e^\top |\bm{T}(\bfx)|^\top \bm{\pi}: \right.\notag\\
&\left.\|\bxi_q-\bfzeta_q^j\|_\infty\leq \theta, \bm{W}^\top \bm{\pi}= \bm{Q}\bxi_q+\bm{q}  \right\}.\label{D_sp_svm412}
\end{align} 
Above, $\bxi_q,\bfzeta_q^j\in\{0,1\}^{m_1}$ and $\|\bxi_q-\bfzeta_q^j\|_\infty\leq \theta $ implies that if $\theta\geq 1$, then $\bxi_q\in \{0,1\}^{m_1}$; otherwise, $\bxi_T=\bfzeta_T^j$. Thus, there are two sub-cases.
\begin{enumerate}[label=(\alph*)]
\item If $\theta<1$, then \eqref{D_sp_svm412} becomes
\begin{align}
	\Z(\bm{x})= \frac{1}{N}\sum_{j\in [N]}\sup_{ \bm{\pi}\in \Re^{\ell}_+}&\left\{(\bm{h}(\bm{x})-\bm{T}(\bfx)\bfzeta^j_T)^\top\bm{\pi}+\theta\e^\top |\bm{T}(\bfx)|^\top \bm{\pi}: \bm{W}^\top \bm{\pi}= \bm{Q}\bfzeta^j_q+\bm{q}  \right\}.\label{D_sp_svm413}
\end{align} 
Let $\bm{y}$ denote the dual variables of constraints $\bm{W}^\top \bm{\pi}= \bm{Q}\bfzeta^j_q+\bm{q}$. Then according to the strong duality of linear programming, we arrive at the first part of \eqref{D_sp_svm812};
\item If $\theta\geq1$, then \eqref{D_sp_svm412} becomes
\begin{align}
	\Z(\bm{x})= \frac{1}{N}\sum_{j\in [N]}\sup_{ \bm{\pi}\in \Re^{\ell}_+,\bxi_q\in \{0,1\}^{m_1}}\left\{(\bm{h}(\bm{x})-\bm{T}(\bfx)\bfzeta^j_T)^\top\bm{\pi}+\theta\e^\top |\bm{T}(\bfx)|^\top \bm{\pi}: \bm{W}^\top \bm{\pi}= \bm{Q}\bxi_q+\bm{q}  \right\},\label{D_sp_svm414}
\end{align} 
Since the constraint system in \eqref{D_sp_svm414} is assumed to be integral, thus \eqref{D_sp_svm414} is equivalent to its continuous relaxation
\begin{align}
	\Z(\bm{x})= \frac{1}{N}\sum_{j\in [N]}\sup_{\bm{\pi}\in \Re^{\ell}_+,\bxi_q\in [0,1]^{m_1}}\left\{(\bm{h}(\bm{x})-\bm{T}(\bfx)\bfzeta^j_T)^\top\bm{\pi}+\theta\e^\top |\bm{T}(\bfx)|^\top \bm{\pi}: \bm{W}^\top \bm{\pi}= \bm{Q}\bxi_q+\bm{q}  \right\},\label{D_sp_svm415}
\end{align}
Let $\bm{y}$ denote the dual variables of constraints $\bm{W}^\top \bm{\pi}= \bm{Q}\bfzeta^j_q+\bm{q}$. Then according to strong duality of linear programming, we arrive at the second part of \eqref{D_sp_svm812}.
\end{enumerate} 
\end{enumerate}
  \qedA
\end{proof}
\end{subequations}

We make the following remarks about \Cref{D_thm_tractable1} and its corresponding formulations \eqref{D_sp_svm81} and \eqref{D_sp_svm812}.
\begin{enumerate}[label=(\roman*)]
\item We can introduce auxiliary variables to linearize the terms $ \|\bm{Q}^\top\bm{y}\|_1,|\T(\bm x)|,(\bm{Q}^\top\bfy)_+$ and reformulate the minimization problems \eqref{D_sp_svm81} and \eqref{D_sp_svm812} as linear programs;
\item If the assumption that $\T(\bfx)\in \Re_+^{\ell\times m_1}$ or $\T(\bfx)\in \Re_-^{\ell\times m_1}$ does not hold, then \eqref{D_sp_svm81} provides an upper bound for $\Z(\bfx)$ and this upper bound will become exact when $\theta\rightarrow 0$; and
\item If one of assumptions that (1) $\T(\bfx)\in \Re_+^{\ell\times m_1}$ or $\T(\bfx)\in \Re_-^{\ell\times m_1}$; and (2) the polyhedron 
$\left\{(\bm\pi,\bm\xi_q)\in \Re_+^\ell\times [0,1]^{m_1}: \bm{W}^\top \bm{\pi}= \bm{Q}\bxi_q+\bm{q}\right\}$
is integral, does not hold, then \eqref{D_sp_svm812} provides an upper bound for $\Z(\bfx)$ and this upper bound will become exact when $\theta\rightarrow 0$.
\end{enumerate}

According to the representation results in \Cref{D_thm_tractable1}, we provide the following equivalent deterministic reformulation of DRTSP \eqref{sp}. 
\begin{proposition}\label{D_thm_tractable1_deterministic}
Suppose $p=\infty$, and $\T(\bfx)\in \Re_+^{\ell\times m_1}$ or $\T(\bfx)\in \Re_-^{\ell\times m_1}$.
\begin{enumerate}[label=(\roman*)]
\item If $\Xi=\Re^{m_1}\times \{0,1\}^{m_2}$, then DRTSP \eqref{sp} is equivalent to 
\begin{subequations}\label{eq_1_tractable1_deterministic}\everymath{\displaystyle}
\begin{align}
v^*= \min_{\bm{x},\bm{y}}\quad& \bm{c}^{\top}\bm{x}+\frac{1}{N}\sum_{j\in [N]}[(\bm{Q}\bfzeta_q^j+\bm{q})^\top \bm{y}^j+\theta \|\bm{Q}^\top\bm{y}^j\|_1],\\
\textrm{s.t.   }	\quad&\displaystyle{\begin{array}{>{\displaystyle}l>{\displaystyle}l}
\displaystyle-(-\bm{T}(\bfx))\displaystyle_+\e+\W\bm{y}^j\geq \bm{h}(\bm{x}),\forall j\in [N],&\text{ if } \theta\geq 1\\
\displaystyle\bm{T}(\bfx)\bfzeta^j_T+\W\bm{y}^j\geq \bm{h}(\bm{x}) ,\forall j\in [N],&\text{ if } \theta<1
\end{array}},\\
&\bm{x}\in \X,\bm{y}^j\in \Re^{n_2},\forall j\in [N].
	\end{align}
\end{subequations}

\item If $\Xi=\{0,1\}^{m_1}\times \Re^{m_2}$ and
the polyhedron 
$\left\{(\bm\pi,\bm\xi_q)\in \Re_+^\ell\times [0,1]^{m_1}: \bm{W}^\top \bm{\pi}= \bm{Q}\bxi_q+\bm{q}\right\}$
is integral, then RTSP \eqref{sp} is equivalent to 
\begin{subequations}
\begin{align}
v^*= \min_{\bm{x},\bm{y},\bm{\sigma}}\quad& \bm{c}^{\top}\bm{x}+\frac{1}{N}\sum_{j\in [N]}[(\bm{Q}\bfzeta_q^j+\bm{q})^\top \bm{y}^j+\I(\theta>1)\e^\top(\bm{Q}^\top\bfy^j)_+],\\
\textrm{s.t.   }	\quad&\bm{T}(\bfx)\bfzeta^j_T+\W\bm{y}^j-\theta|\bm{T}(\bfx)|\e\geq \bm{h}(\bm{x}),\forall j\in [N],\\
&\bm\sigma^j\geq \bm{Q}^\top \bfy^j,\forall j\in [N] ,\\
&\bm{x}\in \X,\bm{y}^j\in \Re^{n_2},\forall j\in [N].
	\end{align}
\end{subequations}

\end{enumerate}
%
\end{proposition}

We next illustrate the proposed formulation \eqref{eq_1_tractable1_deterministic} using \Cref{example1}, where we realize the fact that support of disruption risks is binary, i.e., $\tilde{\bm \delta}\in \{0,1\}^{n_1}$.
\begin{example}\label{example4}\rm Following the notation in \Cref{example1}, let us consider DR-RFLP with both demand and disruption uncertainties. We further suppose that the reference distance is $\|\cdot\|_{\infty}$ and the support of $\tilde{\bm{\xi}}$ is $\{0,1\}^{n_1}\times \Re^\ell$. Since the coefficients of uncertain parameters in the constraints \eqref{RFLP_sp-det_2nd2} have the same sign, according to \Cref{D_thm_tractable1_deterministic}, DR-RFLP can be equivalently formulated as the following MILP:
\begin{subequations}\label{DR-FRLP4}
\begin{align}
v^*= \min_{\bm{x},\bm{y}} & \ \ \bm{f}^{\top}\bm{x}+\frac{1}{N}\sum_{j\in [N]}\sum_{t\in [\ell]}\sum_{s\in [n_1+1]}c_{ts}(\hat{d}_t^j +\theta)y_{ts}^j, \label{RFLP_sp-obj4} \\
\rm{s.t. } 
& \ \ \sum_{s\in [n_1+1]}y_{ts}^j=1, \forall j\in [N],\forall t\in [\ell],\\
&\ \ y_{ts}^j \leq \I(\theta<1)\hat{\delta}_s^jx_s, \forall j\in [N],\forall t\in [\ell],\forall s\in [n_1], \label{RFLP_sp-det_2nd24}\\
& \ \ \bm{x} \in \{0,1\}^{n_1}, \bm{y}^j\in \Re_+^{\ell\times n_1},\forall j\in [N]. \label{RFLP_sp-det4}
\end{align}
\end{subequations}
Clearly, formulation \eqref{DR-FRLP4} is less conservative than \eqref{DR-FRLP1}, since the right-hand sides of constraints \eqref{RFLP_sp-det_2nd24} are no smaller than those in \eqref{RFLP_sp-det_2nd21}. This demonstrates that exploring binary support can indeed help reduce the conservatism of the distributionally robust models.\QEDA

\end{example}

\subsection{Tractable Reformulation II: With Objective Uncertainty Only}
Unlike \Cref{thm_tractable1_AC}, in general, we cannot provide tractable reformulations for the DRTSP with only binary objective uncertainty, and its complexity analysis is postponed to \Cref{D_sec_compl}. Instead, we provide a special case where the tractable reformulation can be derived.
\begin{theorem}\label{D_thm_tractable1_AC}
Suppose that $\Xi=\{0,1\}^{m_1}\times\{\bxi_T\}$ and the polyhedron 
$$\left\{(\bm\pi,\bm\xi_q)\in \Re_+^\ell\times [0,1]^{m_1}: \bm{W}^\top \bm{\pi}= \bm{Q}\bxi_q+\bm{q}, \sum_{t\in \C_0(\bfzeta_q^j)}\xi_{qt}+\sum_{t\in \C_1(\bfzeta_q^j)}(1-\xi_{qt})\leq \kappa\right\}$$
is integral for all $j\in [N]$ and integer $\kappa\in \Ze_+$, where sets $\C_0(\bfzeta_q^j):=\{t\in [m_1]:\zeta_{qt}^j=0 \}$ and $\C_1(\bm{\zeta}_q^j):=\{t\in [m_1]:\zeta_{qt}^j=1 \}$. Then for any $p\in [1,\infty)$, the function $\Z(\bm{x})$ is equivalent to 
\begin{align}
	\Z(\bm{x})= \frac{1}{N}\sum_{j\in [N]}&\min_{ \bm{y}\in \Re^{n_2},\lambda\in \Re_+,\bm\sigma\in \Re_+^{m_1}}	\left\{(\bm{Q}\bfzeta_q^j+\bm{q})^\top \bm{y}+\lfloor\theta^p\rfloor\lambda+\e^\top\bm\sigma:\bm{T}(\bfx)\bxi_T+\W\bm{y}\geq \bm{h}(\bm{x}),\right.\notag\\
&\left.  \lambda+\sigma_t\geq (\bm{Q}^\top \bfy)_t, \forall t\in\C_0(\bfzeta_q^j), \lambda+\sigma_t\geq -(\bm{Q}^\top \bfy)_t,\forall t\in\C_1(\bfzeta_q^j)\right\}.\label{D_sp_svm8_AC}
	\end{align}
\end{theorem}
\begin{subequations}
\begin{proof}
Since $p\in [1,\infty)$ and $\Xi=\{0,1\}^{m_1}\times\{\bxi_T\}$, thus \eqref{sp_svm2} becomes
\begin{align}
	&\Z(\bm{x})= \frac{1}{N}\sum_{j\in [N]}\sup_{ \bm{\pi}\in \Re^{\ell}_+,\bxi_q\in \{0,1\}^{m_1}}\left\{(\bm{h}(\bm{x})-\bm{T}(\bfx)\bxi_T)^\top\bm{\pi}:\|\bxi_q-\bfzeta_q^j\|_p\leq \theta, \bm{W}^\top \bm{\pi}= \bm{Q}\bxi_q+\bm{q} \right\}.\label{D_sp_svm41_AC}
\end{align} 
Since both $\bxi_q,\bfzeta_q^j\in \{0,1\}^{m_1}$, let sets $\C_0(\bfzeta_q^j):=\{t\in [m_1]:\zeta_{qt}^j=0 \}$ and $\C_1(\bm{\zeta}_q^j):=\{t\in [m_1]:\zeta_{qt}^j=1 \}$. Therefore, we have the following linearization results:
\begin{align}
\|\bxi_q-\bfzeta_q^j\|_p^p=\sum_{t\in \C_0(\bfzeta_q^j)}\xi_{qt}+\sum_{t\in \C_1(\bfzeta_q^j)}(1-\xi_{qt}).\label{eq_linearization}
\end{align}

Thus, \eqref{D_sp_svm41_AC} becomes
\begin{align}
	\Z(\bm{x})= \frac{1}{N}\sum_{j\in [N]}&\sup_{ \bm{\pi}\in \Re^{\ell}_+,\bxi_q\in \{0,1\}^{m_1}}\left\{(\bm{h}(\bm{x})-\bm{T}(\bfx)\bxi_T)^\top\bm{\pi}:\right.\notag\\
&\left.\sum_{t\in \C_0(\bfzeta_q^j)}\xi_{qt}+\sum_{t\in \C_1(\bfzeta_q^j)}(1-\xi_{qt})\leq \lfloor\theta^p\rfloor, \bm{W}^\top \bm{\pi}= \bm{Q}\bxi_q+\bm{q} \right\}.\label{D_sp_svm41_AC2}
\end{align}
Since the constraint system of the inner supremum \eqref{D_sp_svm41_AC2} is integral according to our assumption, thus, we can relax the binary variables to be continuous. Thus, we have 
\begin{align}
	\Z(\bm{x})= \frac{1}{N}\sum_{j\in [N]}&\sup_{ \bm{\pi}\in \Re^{\ell}_+,\bxi_q\in [0,1]^{m_1}}\left\{(\bm{h}(\bm{x})-\bm{T}(\bfx)\bxi_T)^\top\bm{\pi}:\right.\notag\\
&\left.\sum_{t\in \C_0(\bfzeta_q^j)}\xi_{qt}+\sum_{t\in \C_1(\bfzeta_q^j)}(1-\xi_{qt})\leq \lfloor\theta^p\rfloor, \bm{W}^\top \bm{\pi}= \bm{Q}\bfzeta^j_q+\bm{Q}(\bxi_q-\bfzeta^j_q)+\bm{q} \right\}.\label{D_sp_svm41_AC3}
\end{align}
Let $\bm{y}$ denote the dual variables of the constraints $\bm{W}^\top \bm{\pi}= \bm{Q}\bfzeta^j_q+\bm{Q}(\bxi_q-\bfzeta^j_q)+\bm{q} $, $\lambda$ be the dual variable of constraint $\sum_{t\in \C_0(\bfzeta_q^j)}\xi_{qt}+\sum_{t\in \C_1(\bfzeta_q^j)}(1-\xi_{qt})\leq \lfloor\theta^p\rfloor$, and $\bm{\sigma}$ be the dual variables of constraints $\bxi_q\leq \e$. Then according to the strong duality of linear programming, \eqref{D_sp_svm41_AC3} is equivalent to
\eqref{D_sp_svm8_AC}.
  \qedA
\end{proof}
\end{subequations}

We make the following remarks about \Cref{D_thm_tractable1_AC} and its corresponding formulation \eqref{D_sp_svm8_AC}.
\begin{enumerate}[label=(\roman*)]
\item Clearly, since DRTSP with objective uncertainty only is a special case of general DRTSP, thus the result of \Cref{D_thm_tractable1} directly follows and is not listed here;
\item The penalty term $\lfloor\theta^p\rfloor\lambda+\e^\top\bm\sigma$ with auxiliary variables $\lambda,\bm{\delta}$ is used to enforce the robustness of the formulation. This penalty term becomes
\[\sum_{t\in \C_0(\bfzeta_q^j)}((\bm{Q}^\top\bfy)_t)_++\sum_{t\in \C_1(\bfzeta_q^j)}(-(\bm{Q}^\top\bfy)_t)_+\]
if $\theta^p\geq m_1$; and 
\item If the integrality assumption of the polyhedra in \Cref{D_thm_tractable1_AC} does not hold, then \eqref{D_sp_svm8_AC} provides an upper bound for the function $\Z(\bfx)$ and this upper bound will become exact when $\theta\rightarrow 0$.
\end{enumerate}

According to the representation results in \Cref{D_thm_tractable1_AC}, we provide the following equivalent deterministic reformulation of DRTSP \eqref{sp}. 
\begin{proposition}\label{D_thm_tractable1_deterministic_AC}
Suppose that $\Xi=\{0,1\}^{m_1}\times\{\bxi_T\}$ and the polyhedron 
$$\left\{(\bm\pi,\bm\xi_q)\in \Re_+^\ell\times [0,1]^{m_1}: \bm{W}^\top \bm{\pi}= \bm{Q}\bxi_q+\bm{q}, \sum_{t\in \C_0(\bfzeta_q^j)}\xi_{qt}+\sum_{t\in \C_1(\bfzeta_q^j)}(1-\xi_{qt})\leq \kappa\right\}$$
is integral for all $j\in [N]$ and integer $\kappa\in \Ze_+$, where sets $\C_0(\bfzeta_q^j):=\{t\in [m_1]:\zeta_{qt}^j=0 \}$ and $\C_1(\bm{\zeta}_q^j):=\{t\in [m_1]:\zeta_{qt}^j=1 \}$. Then for any $p\in [1,\infty)$, DRTSP \eqref{sp} is equivalent to 
\begin{subequations}
\begin{align}
v^*= \min_{\bm{x},\bm{y}}\quad& \bm{c}^{\top}\bm{x}+\frac{1}{N}\sum_{j\in [N]}[(\bm{Q}\bfzeta_q^j+\bm{q})^\top \bm{y}^j+\lfloor\theta^p\rfloor\lambda^j+\e^\top\bm\sigma^j],\\
\textrm{s.t.   }	\quad&\bm{T}(\bfx)\bxi_T+\W\bm{y}^j\geq \bm{h}(\bm{x}),\forall j\in [N],\\
&\lambda^j+\sigma_t^j\geq (\bm{Q}^\top \bfy^j)_t,\forall j\in [N], \forall t\in\C_0(\bfzeta_q^j),\\
& \lambda^j+\sigma_t^j\geq -(\bm{Q}^\top \bfy^j)_t,\forall j\in [N],\forall t\in\C_1(\bfzeta_q^j),\\
&\bm{x}\in \X,\bm{y}^j\in \Re^{n_2},\lambda^j,\bm{\sigma}^j\in \Re^{m_1},\forall j\in [N].
	\end{align}
\end{subequations}

\end{proposition}

\subsection{Tractable Reformulation III: With Constraint Uncertainty Only}

Similarly, we provide special cases of DRTSP with only binary constraint uncertainty such that the tractable reformulations can be derived.
\begin{theorem}\label{D_thm_approximation1} 
Suppose that $\Xi= \{\bxi_q\}\times \{0,1\}^{m_2}$, $p\in [1,\infty)$, and $\theta\in [1,\sqrt[p]{2})$. Then the function $\Z(\bm{x})$ is equivalent to 
\begin{align}
	&\Z(\bm{x})= \frac{1}{N}\sum_{j\in [N]}\max_{i\in [m_2+1]}\min_{ \bm{y}\in \Re^{n_2}}	\left\{(\bm{Q}\bxi_q^j+\bm{q})^\top \bm{y}:\bm{T}(\bfx)\hat{\bfzeta}^{ij}_T+\W\bm{y}\geq \bm{h}(\bm{x}) \right\},\label{D_sp_svm5}
	\end{align}
where for each $i\in [m_2+1]$ and ,
\begin{align}\label{def_hat_zeta}
\hat{\bfzeta}^{ij}_T=\bfzeta^{j}_T+
\begin{cases}
\bm{0}, &\text{ if }i=m_2+1\\
\e_i, &\text{ if }i\in \C_0(\bfzeta_T^j)\\
-\e_i, &\text{ if }i\in \C_1(\bfzeta_T^j)
\end{cases},
\end{align}
and sets $\C_0(\bfzeta_T^j):=\{t\in [m_2]:\zeta_{Tt}^j=0 \}$ and $\C_1(\bm{\zeta}_T^j):=\{t\in [m_2]:\zeta_{Tt}^j=1 \}$.
\end{theorem}
\begin{subequations}
\begin{proof}
Since $p\in [1,\infty)$ and $\Xi=\{\bxi_q\}\times \{0,1\}^{m_2}$, \eqref{sp_svm2} becomes
\begin{align}
	&\Z(\bm{x})= \frac{1}{N}\sum_{j\in [N]}\sup_{ \bm{\pi}\in \Re^{\ell}_+,\bxi_T\in\{0,1\}^{m_2}}\left\{(\bm{h}(\bm{x})-\bm{T}(\bfx)\bxi_T)^\top\bm{\pi}:\|\bxi_T-\bfzeta_T^j\|_p\leq \theta, \bm{W}^\top \bm{\pi}= \bm{Q}\bxi_q+\bm{q} \right\},\label{D_sp_svm3_AO}
	\end{align}

According to \eqref{eq_linearization}, and the fact that $\theta\in [1,\sqrt[p]{2})$, we know that
\begin{align*}
\left\{\bxi_T\in\{0,1\}^{m_2}:\|\bxi_T-\bfzeta_T^j\|_p\leq \theta\right\}=\{\bm{0}\}\cup\{\bfzeta^{j}_T+\e_i\}_{i\in \C_0(\bfzeta_T^j)}\cup \{\bfzeta^{j}_T-\e_i\}_{i\in \C_1(\bfzeta_T^j)}:=\{\hat{\bfzeta}^{ij}_T\}_{i\in [m_2+1]}
\end{align*}
Hence, optimizing $\bxi_T$ first, we arrive at
\begin{align}
	&\Z(\bm{x})=\frac{1}{N}\sum_{j\in [N]}\max_{i\in [m_2+1]}\sup_{ \bm{\pi}\in \Re^{\ell}_+}\left\{(\bm{h}(\bm{x})-\bm{T}(\bfx)\hat{\bfzeta}_T^{ij})^\top\bm{\pi}:\bm{W}^\top \bm{\pi}= \bm{Q}\bxi_q+\bm{q} \right\},\label{D_sp_svm4_AO}
	\end{align} 
Taking the dual of inner supremum and using strong duality of linear programming, we arrive at \eqref{D_sp_svm5}.
  \qedA
\end{proof}
\end{subequations}

We make the following remarks about \Cref{D_thm_approximation1} and its corresponding formulation \eqref{D_sp_svm5}.
\begin{enumerate}[label=(\roman*)]
\item To evaluate the function $\Z(\bfx)$, one needs to solve $m_1+1$ linear programs for each $j\in [N]$;
\item If $\theta\in [0,1)$, then according to the proof of \Cref{D_thm_approximation1},
\begin{align}
	&\Z(\bm{x})= \frac{1}{N}\sum_{j\in [N]}\min_{ \bm{y}\in \Re^{n_2}}	\left\{(\bm{Q}\bxi_q^j+\bm{q})^\top \bm{y}:\bm{T}(\bfx)\bfzeta^{j}_T+\W\bm{y}\geq \bm{h}(\bm{x}) \right\},\label{D_sp_svm52}
	\end{align}
i.e., the function $\Z(\bfx)$ is equivalent to its sampling average approximation counterpart.
\end{enumerate}

Below provides an equivalent deterministic reformulation of DRTSP \eqref{sp}.
\begin{proposition}\label{D_thm_approximation1_deterministic}
Suppose that $\Xi= \{\bxi_q\}\times \{0,1\}^{m_2}$, $p\in [1,\infty)$, and $\theta\in [1,\sqrt[p]{2})$. Then DRTSP \eqref{sp} is equivalent to 
\begin{subequations}
\begin{align}
v^*= \min_{\bm{x},\bm{\eta}}\quad& \bm{c}^{\top}\bm{x}+\frac{1}{N}\sum_{j\in [N]}\eta_j,\\
\textrm{s.t.   }	\quad&\eta_{j}\geq (\bm{Q}\bxi_q+\bm{q})^\top \bm{y}^{ij} ,\forall j\in [N],\forall i\in [m_2+1],\\
&\bm{T}(\bfx)\hat{\bfzeta}_T^{ij}+\W\bm{y}^{ij}\geq \bm{h}(\bm{x}), \forall j\in [N],\forall i\in [m_2+1],\\
&\bm{x}\in \X,\bm{y}^{ij}\in \Re^{n_2},\forall j\in [N],i\in [m_2+1],
	\end{align}
\end{subequations}
where $\{\hat{\bfzeta}_T^{ij}\}_{i\in [m_2+1],j\in [N]}$ are defined in \eqref{def_hat_zeta}
\end{proposition}

We note that if the number of the extreme points of dual constraint system of \eqref{recourse} is small, then equivalently, we can represent the recourse function in the form of piece-wise max of affine functions in the random parameters, and the tractable reformulation can be extended to the case with any reference distance $\|\cdot\|_p$ such that $p\in [1,\infty)$.
\begin{proposition}\label{D_prop_max}Suppose that $\Xi=\{0,1\}^{\tau}, p\in [1,\infty),$ and $z(\bm{x})=\sup_{\Pr\in \P}\E_{\Pr}[\max_{i\in [m]}\{\bm{a}_i(\bm{x})^\top \bm{\xi}+d_i(\bm{x})\}]$ with affine functions $\bm{a}_i(\bm{x}):\Re^{n_1}\rightarrow\Re^\tau$ and $d_i(\bm{x}):\Re^{n_1}\rightarrow\Re$ for each $i\in [m]$. Then 
\begin{itemize}
\item  Function $z(\bm{x})$ is equivalent to
\begin{align}\label{D_eq_rhs_uncertainty}
z(\bm{x})=\frac{1}{N}\sum_{j\in [N]}\max_{i\in [m]}\sup_{ \bxi\in [0,1]^{\tau}}\left\{\bm{a}_i(\bm{x})^\top \bm{\xi}+d_i(\bm{x}):\sum_{t\in \C_0(\bfzeta^j)}\xi_t+\sum_{t\in \C_1(\bfzeta^j)}(1-\xi_t)\leq \lfloor\theta^p\rfloor\right\},
\end{align}
where sets $\C_0(\bfzeta^j):=\{t\in [\tau]:\zeta^j_t=0 \}$ and $\C_1(\bm{\zeta}^j):=\{t\in [\tau]:\zeta^j_t=1 \}$; and
\item DRTSP \eqref{sp} is equivalent to 
\begin{subequations}\label{D_eq_reform_drtsp}
\begin{align}
v^*= \min_{\bm{x},\bm{\eta},\lambda,\bm{\sigma}}\quad& \bm{c}^{\top}\bm{x}+\frac{1}{N}\sum_{j\in [N]}\eta_j,\\
\textrm{s.t.   }	\quad&\eta_j\geq \lambda^{ij}\lfloor\theta^p\rfloor+\bm{a}_i(\bm{x})^\top \bm{\zeta}^j+d_i(\bm{x}),\forall j\in [N],\forall i\in [m],\\
&\lambda^{ij}+\sigma_t^{ij}\geq a_{it}(\bm{x}),\forall j\in [N],\forall i\in [m],\forall t\in \C_0(\bfzeta^j),\\
&\lambda^{ij}+\sigma_t^{ij}\geq -a_{it}(\bm{x}),\forall j\in [N],\forall i\in [m], \forall t\in \C_1(\bfzeta^j),\\
&\bm{x}\in \X,\lambda^{ij}\in \Re_+,\bm\sigma^{ij}\in \Re_+^\tau,\forall j\in [N],\forall i\in [m] .
	\end{align}
\end{subequations}
\end{itemize}

\end{proposition}
\begin{subequations}
\begin{proof}
Since $\Xi=\{0,1\}^{\tau}$ and $Z(\bfx,\bxi)=\max_{i\in [m]}\{\bm{a}_i(\bm{x})^\top \bm{\xi}+d_i(\bm{x})\}$, \eqref{sp_svm1} becomes
\begin{align}
	&\Z(\bm{x})= \frac{1}{N}\sum_{j\in [N]}\max_{i\in [m]}\sup_{ \bxi\in \{0,1\}^{\tau}}\left\{\bm{a}_i(\bm{x})^\top \bm{\xi}+d_i(\bm{x}):\|\bxi-\bfzeta^j\|_p\leq \theta\right\}.\label{D_sp_svm31}
	\end{align} 
According to \eqref{eq_linearization}, \eqref{D_sp_svm31} becomes
\begin{align}
	&\Z(\bm{x})= \frac{1}{N}\sum_{j\in [N]}\max_{i\in [m]}\sup_{ \bxi\in \{0,1\}^{\tau}}\left\{\bm{a}_i(\bm{x})^\top \bm{\xi}+d_i(\bm{x}):\sum_{t\in \C_0(\bfzeta^j)}\xi_t+\sum_{t\in \C_1(\bfzeta^j)}(1-\xi_t)\leq \lfloor\theta^p\rfloor\right\}.\label{D_sp_svm32}
	\end{align}
Since the feasible region defined by cardinality constraint is integral, thus, we can relax the binary variables in the inner supremum of \eqref{D_sp_svm32} to be continuous. Thus, we arrive at \eqref{D_eq_rhs_uncertainty}.

To derive the formulation \eqref{D_eq_reform_drtsp}, let us first take the dual of inner supremum with dual variables $\lambda, \bm{\sigma}$ and use the strong duality of linear programming. Thus, \eqref{D_eq_rhs_uncertainty} is equivalent to
\begin{align}
\Z(\bm{x})= \frac{1}{N}\sum_{j\in [N]}\max_{i\in [m]}	&\min_{\lambda\in \Re_+,\bm{\sigma}\in \Re_+^{m_2}}\left\{\lambda\lfloor\theta^p\rfloor+\bm{a}_i(\bm{x})^\top \bm{\zeta}^j+d_i(\bm{x}):\right.\notag\\
&\left.\lambda+\sigma_t\geq a_{it}(\bm{x}),\forall t\in \C_0(\bfzeta^j),\lambda+\sigma_t\geq -a_{it}(\bm{x}), \forall t\in \C_1(\bfzeta^j)\right\}.\label{D_sp_svm33}
	\end{align}

Then the conclusion follows from a straightforward linearization.
  \qedA
\end{proof}
\end{subequations}

We will illustrate the proposed formulation in \Cref{D_thm_approximation1_deterministic} using \Cref{example1}, where we consider that there is no demand uncertainty, i.e., the only uncertain parameters are facility disruptions, and the support of random disruptions is $\{0,1\}^{n_1}$.
\begin{example}\label{example5}\rm Following the notation in \Cref{example1}, let us consider DR-RFLP with only disruption risks, i.e., the demand is deterministic satisfying $\Pr\{\tilde{\bm d}=\bm d\}=1$.

Suppose the reference distance is $\|\cdot\|_{1}$, the support of $\tilde{\bm{\xi}}$ is $\{0,1\}^{n_1}\times \{\bm d\}$, and
the Wasserstein radius $\theta\in [1,\sqrt[p]{2})$. According to \Cref{D_thm_approximation1_deterministic}, DR-FRLP with disruption risks can be equivalently formulated as the following MILP:
\begin{subequations}\label{DR-FRLP15}
\begin{align}
v^*= \min_{\bm{x},\bm{y}} & \ \ \bm{f}^{\top}\bm{x}+\frac{1}{N}\sum_{j\in [N]}\eta_j, \label{RFLP_sp-obj15} \\
\rm{s.t. } 
&\ \  \eta_j\geq \sum_{t\in [\ell]}\sum_{s\in [n_1+1]}c_{ts}d_t y_{ts}^{ij}, \forall j\in [N],\forall i\in [n_1+1],\\
& \ \ \sum_{s\in [n_1+1]}y_{ts}^{ij}=1, \forall j\in [N],\forall t\in [\ell],\forall i\in [n_1+1],\\
&\ \ y_{ts}^{ij} \leq \bar{\delta}_s^{ij}x_s, \forall j\in [N],\forall t\in [\ell],\forall i\in [n_1+1],\forall s\in [n_1], \label{RFLP_sp-det_2nd215}\\
& \ \ \bm{x} \in \{0,1\}^{n_1}, \bm{y}^{ij}\in \Re_+^{\ell\times n_1},\forall j\in [N],\forall i\in [n_1+1],. \label{RFLP_sp-det15}
\end{align}
\end{subequations}
where for each $i\in [n_1+1]$ and ,
\begin{align*}
\bar{\bm\delta}^{ij}_T=\hat{\bm\delta}^j+
\begin{cases}
\bm{0}, &\text{ if }i=n_1+1\\
\e_i, &\text{ if }i\in \C_0(\hat{\bm\delta}^j)\\
-\e_i, &\text{ if }i\in \C_1(\hat{\bm\delta}^j)
\end{cases}.
\end{align*}
\QEDA
\end{example}

\subsection{Complexity Analysis}\label{D_sec_compl}
Finally, we close this section by showing that for general reference distance $\|\cdot\|_{p}$ with $p\in[1,\infty]$, either with objective uncertainty only or with constraint uncertainty only, computing the function $\Z(\bm{x})$ with $N=1$ can be NP-hard.
	\begin{restatable}{proposition}{propcompexltiyD}\label{prop_compexltiy_D}
Computing $\Z(\bm{x})$ is NP-hard for any $p\in [1,\infty]$ whenever
\begin{enumerate}[label=(\roman*)]
\item (Without Constraint Uncertainty)  $N=1$, $\Xi=\{0,1\}^{m_1}\times \{\bxi_T\}$, $\bm{h}(\bm{x})=\bm{0},\bm{T}(\bfx)=\bm{0}$, and Wasserstein radius $\theta\geq \sqrt[p]{m_1}$; or
\item (Without Objective Uncertainty) $N=1$, $\Xi=\{\bxi_T\}\times \{0,1\}^{m_2} $, $\bm{h}(\bm{x})=\bm{0},\bm{T}(\bfx)=\textrm{const.}$, and Wasserstein radius $\theta\geq\sqrt[p]{m_2}$.
\end{enumerate}
\end{restatable}
\begin{proof}Let us first consider the NP-complete problem - feasibility problem of a general binary program which asks
\begin{quote}
	{(Feasibility problem of a general binary program)} Given a rational matrix $\bm{A}\in \Qe^{t_1\times t_2}$ and a rational vector $\bm{b}\in \Qe^{t_1}$, is there exists a binary vector $\bm{r}\in \{0,1\}^{t_2}$ such that $\bm{A}\bm{r}=\bm{b}$?
\end{quote}

Next, we split the proof into two cases- when $\Xi=\{0,1\}^{m_1}\times\{\bxi_T\}$ and when $\Xi=\{\bxi_T\}\times \{0,1\}^{m_2}$.
\begin{enumerate}[label=(\roman*)]
\item When $N=1,\Xi=\{0,1\}^{m_1}\times\{\bxi_T\}$, let $\bm{h}(\bm{x})=\bm{0},\bm{T}(\bfx)=\bm{0}$, $\bm{W}^\top=\begin{bmatrix}
\bm{A}\\
\bm{I}_{t_2}
\end{bmatrix}, \bm{Q}=\begin{bmatrix}
\bm{0}\\
\bm{I}_{t_2}
\end{bmatrix} ,\q=\begin{pmatrix}
\bm{b}\\
\bm{0}
\end{pmatrix},\bm{\pi}=\bm{r}$, and $\ell=t_2,m_1=t_2,n_2=t_1+t_2$. As $\theta\geq \sqrt[p]{m_1}$, thus \eqref{sp_svm2} becomes
\begin{align}
	&\Z(\bm{x})=\sup_{ \bm{r}\in \Re^{t_2}_+,\bxi_q}\left\{0:\bxi_q\in \{0,1\}^{t_2}, \bm{A}\bm{r}=\bm{b},\bm{r}=\bxi_q \right\}.\label{D_sp_svm41_AC_comp}
\end{align} 
Clearly, $\Z(\bm{x})=0$ if and only if the binary program $\{\bm{r}\in \{0,1\}^{t_2}:\bm{A}\bm{r}=\bm{b}\}$ is feasible.

 \item When $N=1,\Xi=\{\bxi_T\}\times \{0,1\}^{m_2}$, let $\bm{h}(\bm{x})=\bm{0},\bm{T}(\bfx)=\allowbreak[
\e_{t_2+1}-\e_1,\ldots, \e_{2t_2}-\e_{t_2},\e_{1}-\e_{t_2+1},\ldots,\allowbreak\e_{t_2}-\e_{2t_2}]$,  $\bm{W}^\top=\begin{bmatrix}
\bm{A}&\bm{0}\\
\bm{I}_{t_2}&\bm{I}_{t_2}
\end{bmatrix},\bm{Q}=\bm{0}, \q=\begin{pmatrix}
\bm{b}\\
\e
\end{pmatrix},\bm{\pi}=\begin{pmatrix}
\bm{r}\\
\bm{s}
\end{pmatrix}$, and $\ell=2t_2,m_2=2t_2,n_2=t_1+t_2$. As $\theta\geq \sqrt[p]{m_2}$, thus \eqref{sp_svm2} becomes
\begin{align}
	&\Z(\bm{x})=\sup_{ \bm{r}\in \Re^{t_2}_+,\bm{s}\in \Re^{t_2}_+,\bxi_q}\left\{\sum_{i\in [t_2]} (\xi_{qi}-\xi_{q(t_2+i)})(r_i-s_i):\bxi_q\in \{0,1\}^{2t_2}, \bm{A}\bm{r}=\bm{b},\bm{r}+\bm{s}=\e \right\},
\end{align} 
which is equivalent to
\begin{align}
	&\Z(\bm{x})=\sup_{ \bm{r}\in \Re^{t_2}_+,\bm{s}\in \Re^{t_2}_+}\left\{\sum_{i\in [t_2]} |(r_i-s_i)|: \bm{A}\bm{r}=\bm{b},\bm{r}+\bm{s}=\e \right\}.\label{D_sp_svm41_AO_comp}
\end{align}
Above, $\Z(\bm{x})=t_2$ if and only if there exists a binary vector $(\bm{r},\bm{s})\in \{0,1\}^{m_1}\times \{0,1\}^{m_1}$ such that $\bm{A}\bm{r}=\bm{b},\bm{r}+\bm{s}=\e$. 
Thus, $\Z(\bm{x})=t_2$ if and only if the binary program $\{\bm{r}\in \{0,1\}^{t_2}:\bm{A}\bm{r}=\bm{b}\}$ is feasible.
\end{enumerate}
\qedA
\end{proof}

\section{Summary of Main Results and Formulation Recommendations}\label{sec_rec}

In this section, we provide a summarized 
\Cref{tab:my-table2} for the different formulations studied in \Cref{sec_cont} and \Cref{sec_bin}. For a DRTSP problem, we have the following recommendations about how to choose a proper formulation:

\begin{enumerate}[label=\textbf{Case \arabic{enumi}.},wide = 0pt, itemsep=1.5ex]
\item If all the random parameters in the worst-case wait-and-see problem are continuous, then consider three sub-cases:
\begin{enumerate}[label=\textit{(S1.\arabic*)},wide = 20pt, itemsep=1.5ex]
\item If both objective function and constraint system involve random parameters, then it is better to use reference distance $\|\cdot\|_{\infty}$ and apply the results in \Cref{thm_tractable1} and \Cref{thm_tractable1_deterministic}, which provide tractable formulations if their conditions are satisfied. Otherwise, the these formulations become conservative approximation and will be exact when Wasserstein radius goes to $0$;
\item If only objective involves random parameters, then the results in \Cref{thm_tractable1_AC} and \Cref{thm_tractable1_deterministic_AC} suffice.
\item If random parameters appear only in the constraint system, then it is better to use reference distance $\|\cdot\|_{1}$ and follow the results in \Cref{thm_approximation1} and \Cref{thm_approximation1_deterministic}. In addition, if the recourse function can be expressed as a piecewise maximum of a finite number of affine functions, then the results \Cref{prop_max} apply to any reference distance $\|\cdot\|_p$ with $p\in [1,\infty]$.
\end{enumerate}
\item  If random parameters in the objective function or in the constraint system are binary, then consider three sub-cases:
\begin{enumerate}[label=\textit{(S2.\arabic*)},wide = 20pt, itemsep=1.5ex]
\item If both objective function and constraint system have random parameters, then it is better to use reference distance $\|\cdot\|_{\infty}$ and apply the results in \Cref{D_thm_tractable1} and \Cref{D_thm_tractable1_deterministic}, which provide tractable formulations if their conditions are satisfied. Otherwise, these formulations become conservative approximation and will become exact when Wasserstein radius goes to $0$;
\item If only objective involves random parameters, which are binary, then the results in \Cref{D_thm_tractable1_AC} and \Cref{D_thm_tractable1_deterministic_AC} apply to any reference distance $\|\cdot\|_p$ with $p\in [1,\infty)$ provided that their conditions are satisfied. Similarly, if their conditions are not met, then the results in \Cref{D_thm_tractable1_AC} and \Cref{D_thm_tractable1_deterministic_AC} become conservative approximation;
\item If random parameters appear only in the constraint system, which are binary, then it is better to use reference distance $\|\cdot\|_{1}$ and follow results in \Cref{D_thm_approximation1} and \Cref{D_thm_approximation1_deterministic} given that the Wasserstein radius $\theta$ is small. If the Wasserstein radius is large (i.e., there are very limited empirical data points), then we recommend using results in \Cref{D_thm_tractable1} and \Cref{D_thm_tractable1_deterministic} with reference distance $\|\cdot\|_{\infty}$. In addition, if the recourse function can be expressed as a piecewise maximum of a finite number of affine functions, then the results \Cref{D_prop_max} apply to any reference distance $\|\cdot\|_p$ with $p\in [1,\infty)$.
\end{enumerate}
\end{enumerate}

\begin{table}[]
\caption{}
	\setlength{\tabcolsep}{3.0pt}
\renewcommand{\arraystretch}{1.5}
\label{tab:my-table2}
\centering
\begin{tabular}{|c|c|c|c|c|c|c|}
\hline
\multicolumn{2}{|c|}{\multirow{2}{*}{Support}} & \multicolumn{3}{c|}{Conditions} & \multicolumn{2}{c|}{Formulation} \\ \cline{3-7} 
\multicolumn{2}{|c|}{} & \begin{tabular}[c]{@{}c@{}}Uncertainty\\ Type\end{tabular} & $p$ & \begin{tabular}[c]{@{}c@{}}Other\\ Conditions\\ (Yes/No)\end{tabular} & $\Z(\bfx)$ & DRTSP \\ \hline
\multirow{4}{*}{\rotatebox[origin=c]{270}{Continuous}} & $\Xi=\Re^{m_1}\times \Re^{m_2}$ & General& $p=\infty$ &Yes  & \Cref{thm_tractable1} &\Cref{thm_tractable1_deterministic}  \\ \cline{2-7} 
& $\Xi=\Re^{m_1}\times \{\bxi_T\}$ & \begin{tabular}[c]{@{}c@{}}Objective\\Uncertainty\end{tabular}& $p\in[1,\infty]$ &No  & \Cref{thm_tractable1_AC} &\Cref{thm_tractable1_deterministic_AC} \\ \cline{2-7} 
& $\Xi=\{\bxi_q\}\times \Re^{m_2}$ & \begin{tabular}[c]{@{}c@{}}Constraint\\Uncertainty\end{tabular}& $p=1$ &No  & \Cref{thm_approximation1}  &\Cref{thm_approximation1_deterministic}\\ \cline{2-7} 
& $\Xi=\{\bxi_q\}\times \Re^{m_2}$ & \begin{tabular}[c]{@{}c@{}}Piecewise\\Maximum\end{tabular}& $p\in[1,\infty]$ &No  & \Cref{prop_max} &\Cref{prop_max}\\ \hline\hline
\multirow{4}{*}{\rotatebox[origin=c]{270}{Discrete}}  & $\Xi=\Re^{m_1}\times \{0,1\}^{m_2}$ & General& $p=\infty$ &Yes  & \Cref{D_thm_tractable1} &\Cref{D_thm_tractable1_deterministic}  \\ \cline{2-7} 
& $\Xi=\{0,1\}^{m_1}\times \Re^{m_2}$ & General& $p=\infty$ &Yes  & \Cref{D_thm_tractable1} &\Cref{D_thm_tractable1_deterministic}  \\ \cline{2-7} 
& $\Xi=\{0,1\}^{m_1}\times \{\bxi_T\}$ & \begin{tabular}[c]{@{}c@{}}Objective\\Uncertainty\end{tabular}& $p\in [1,\infty)$ &Yes  & \Cref{D_thm_tractable1_AC} &\Cref{D_thm_tractable1_deterministic_AC} \\ \cline{2-7} 
& $\Xi=\{\bxi_q\}\times \{0,1\}^{m_2}$ & \begin{tabular}[c]{@{}c@{}}Constraint\\Uncertainty\end{tabular}& $p\in[1,\infty)$ &Yes  & \Cref{D_thm_approximation1}  &\Cref{D_thm_approximation1_deterministic}\\ \cline{2-7} 
& $\Xi=\{\bxi_q\}\times \{0,1\}^{m_2}$ & \begin{tabular}[c]{@{}c@{}}Piecewise\\Maximum\end{tabular}& $p\in [1,\infty)$ &No  & \Cref{D_prop_max} &\Cref{D_prop_max}\\\hline
\end{tabular}
\end{table}

Some additional remarks are provided below. If the random parameters in the objective and  constraint system have very different magnitudes, it is better to \textbf{normalize} the empirical data to avoid numerical issues. It is always good to incorporate \textbf{support} information of continuous random parameters into the formulations. In general, incorporating support into the reformulation in  \Cref{sec_cont} and \Cref{sec_bin} can destroy the tractability results. However, in practice, readers are highly recommended to explore support information and reduce the conservatism of DRTSP models.

\section{Numerical Illustration}\label{sec_sep_numerical}

In this section, we present a numerical study to demonstrate the effectiveness of the proposed formulations and also show how to use cross-validation to choose a proper Wasserstein radius $\delta$. 

For the demonstration purpose, we studied two models, i.e, Model \eqref{DR-FRLP1} and Model \eqref{DR-FRLP4} from \Cref{example1} and \Cref{example4}, respectively. We used normalized 49-node instances provided in \cite{Cui10}, and thus in these two models, $\ell=n_1=49$. The fixed cost and coordinates of candidate locations can be found at the following link \allowbreak\url{https://drive.google.com/file/d/11-oc9xX2-tTlSxkNuZhZ-qZlo7xQq80J/view?usp=sharing}. We assumed that disruption happens independently and each location has a probability of $p\in \{0.01,0.05\}$ to be disrupted, i.e., $\Pr\{\tilde{\delta}_i=0\}=p$ and $\Pr\{\tilde{\delta}_i=1\}=1-p$. To ensure the consistency between random vectors $\tilde{\bm\delta}$ and $\tilde{\bm d}$, we normalized $\tilde{\bm d}$ such that for each $t\in [\ell]$ follows i.i.d uniform distribution in the range between 0.05 and 1.0. We also computed the unit transportation cost $c_{ts}=100\times \text{Euclidean distance between locations $t\in [\ell]$ and $j\in[n_1]$}$. Finally, for the emergency facility (i.e., dummy facility), we assumed that its unit transportation cost is $M=10,000$.

To test these two models, we generate $N=100$ samples of $(\tilde{\bm\delta},\tilde{\bm d})$, where the computational results are displayed in Table~\ref{tab:my-table}. In Table~\ref{tab:my-table}, the Wasserstein radius $\theta$ varies from $0$ to $0.18$, where $\theta=0$, both models are reduced to their sampling average approximation counterpart (SAA model) and for each model, we use Opt.Val., Time, and Built Facilities to denote optimal values, computational time, and built facilities output by the model, respectively. To evaluate the robustness of the solution and choose a proper Wasserstein radius, we generated 100 additional samples, evaluated their corresponding objective function values, and computed the 95\% confidence intervals of their mean values, which are displayed in the columns titled ``Confidence Interval". All the tested instances were executed on a MacBook Pro with a 2.80 GHz processor and 16GB RAM with a call of the commercial solver Gurobi (version 7.5, with default settings).

From Table~\ref{tab:my-table}, we see that all the instances can be solved to the optimality within 1 minute, where Model \eqref{DR-FRLP4} takes a slightly shorter time. We see that when $\theta=0$, the SAA model underestimates the costs, where the underestimation mainly comes from the expected transportation costs (i.e., wait-and-see costs). When the Wasserstein radius $\theta$ increases, the total costs of both Model \eqref{DR-FRLP1} and Model \eqref{DR-FRLP4} increase. However, it is seen that for the same $\theta>0$, the total cost of Model \eqref{DR-FRLP4} is significantly smaller than that of Model \eqref{DR-FRLP1}. This demonstrates that exploring support information of random parameters can help reduce the risk of distributional uncertainty. In addition, we also see that the set of built facilities of Model \eqref{DR-FRLP4} does not change when $\theta$ grows to $0.16$. This demonstrates that the first-stage results from SAA can be robust. When the probability of disruptions $p$ increases from $0.01$ to $0.05$, we see that Model \eqref{DR-FRLP1} does not allow to build any facility due to disruptions when $\theta>0$, while Model \eqref{DR-FRLP4} still works and finds appropriate facility locations. This further demonstrates the less conservatism of Model \eqref{DR-FRLP4}.

To choose a proper Wasserstein radius, we suggest to select the smallest $\theta$ such that its corresponding total cost is beyond the confidence interval. For example, when $p=0.01$, the best Wasserstein radii of Model \eqref{DR-FRLP1} and Model \eqref{DR-FRLP4} are $\theta=0.02$, while when $p=0.05$, the best Wasserstein radius of Model \eqref{DR-FRLP4} are $\theta=0.06$.

\begin{table}[htbp]
\centering
	\setlength{\tabcolsep}{3.0pt}
\renewcommand{\arraystretch}{1.5}
\caption{Numerical results of Model \eqref{DR-FRLP1} and Model \eqref{DR-FRLP4} from \Cref{example1} and \Cref{example4}, where $N=100,\ell=n_1=49$.}
\label{tab:my-table}
\scriptsize
\begin{center}
\begin{threeparttable}
\begin{tabular}{|r|r|r|r|l|r|r|r|l|r|}
\hline
\multirow{2}{*}{$p$}&\multirow{2}{*}{$\theta$} & \multicolumn{4}{c|}{Model \eqref{DR-FRLP1}} & \multicolumn{4}{c|}{Model \eqref{DR-FRLP4}} \\ \cline{3-10} 
& & Opt.Val. & Time & Built Facilities & Confidence Interval & Opt.Val. & Time & Built Facilities & Confidence Interval \\ \hline
\multirow{10}{*}{0.01} & 0.00 & 7288.04 & 7.78 & {[}4, 25, 31, 35, 45{]} & {[}7232.33, 7379.25{]} & 7288.04 & 7.74 & \multirow{9}{*}{{[}4, 25, 31, 35, 45{]}} & \multirow{9}{*}{{[}7232.33, 7379.25{]}} \\ \cline{2-8}
 & 0.02 & 7998.84 & 10.72 & \multirow{2}{*}{{[}13, 16, 25, 31{]}} & \multirow{2}{*}{{[}7641.80, 7862.95{]}} & 7453.31 & 7.07 &  &  \\ \cline{2-4} \cline{7-8}
 & 0.04 & 8344.45 & 18.82 &  &  & 7618.58 & 7.59 &  &  \\ \cline{2-8}
 & 0.06 & 8673.85 & 17.64 & \multirow{2}{*}{{[}13, 16, 21, 30, 31{]}} & \multirow{2}{*}{{[}7748.27, 7963.70{]}} & 7783.85 & 8.37 &  &  \\ \cline{2-4} \cline{7-8}
 & 0.08 & 8995.13 & 24.35 &  &  & 7949.12 & 7.26 &  &  \\ \cline{2-8}
 & 0.10 & 9295.40 & 41.71 & \multirow{5}{*}{{[}13, 16, 19, 21, 30, 31{]}} & \multirow{5}{*}{{[}7986.44, 8123.68{]}} & 8113.81 & 13.71 &  &  \\ \cline{2-4} \cline{7-8}
 & 0.12 & 9568.05 & 30.45 &  &  & 8279.66 & 12.70 &  &  \\ \cline{2-4} \cline{7-8}
 & 0.14 & 9846.47 & 30.16 &  &  & 8444.93 & 13.38 &  &  \\ \cline{2-4} \cline{7-8}
 & 0.16 & 10130.65 & 30.44 &  &  & 8610.20 & 15.21 &  &  \\ \cline{2-4} \cline{7-10} 
 & 0.18 & 10420.59 & 30.60 &  &  & 8769.09 & 14.53 & {[}4, 21, 30, 31, 35, 45{]} & {[}7335.91, 7470.26{]} \\ \hline\hline
\multirow{10}{*}{0.05} & 0.00 & 7498.95 & 9.75 & {[}4, 25, 31, 35, 45{]} & {[}7550.51, 7946.76{]} & 7498.95 & 9.90 & \multirow{9}{*}{{[}4, 25, 31, 35, 45{]}} & \multirow{9}{*}{{[}7550.51, 7946.76{]}} \\ \cline{2-8}
 & 0.02 & ----\tnote{1} & 0.78 & {[}{]} & ---- & 7672.22 & 8.57 &  &  \\ \cline{2-8}
 & 0.04 & ---- & 0.81 & {[}{]} & ---- & 7845.49 & 8.04 &  &  \\ \cline{2-8}
 & 0.06 & ---- & 0.78 & {[}{]} & ---- & 8018.76 & 16.38 &  &  \\ \cline{2-8}
 & 0.08 & ---- & 1.34 & {[}{]} & ---- & 8192.03 & 18.62 &  &  \\ \cline{2-8}
 & 0.10 & ---- & 1.33 & {[}{]} & ---- & 8365.30 & 19.51 &  &  \\ \cline{2-8}
 & 0.12 & ---- & 1.29 & {[}{]} & ---- & 8538.57 & 19.77 &  &  \\ \cline{2-8}
 & 0.14 & ---- & 1.30 & {[}{]} & ---- & 8711.84 & 21.04 &  &  \\ \cline{2-8}
 & 0.16 & ---- & 0.85 & {[}{]} & ---- & 8885.11 & 15.68 &  &  \\ \cline{2-10} 
 & 0.18 & ---- & 0.76 & {[}{]} & ---- & 9052.87 & 19.61 & {[}4, 21, 30, 31, 35, 45{]} & {[}7547.08, 7864.01{]} \\ \hline
\end{tabular}
\begin{tablenotes}
\item[1] ---- means that all the customers will be served by  the emergency facility.
\end{tablenotes}
\end{threeparttable}
\end{center}
\end{table}

\section{Conclusion}\label{sec_conclusion}
This paper studies a distributionally robust two-stage stochastic program (DRTSP) with $\infty-$Wasserstein ambiguity set. We provide sufficient conditions under which the worst-case expected wait-and-see cost of DRTSP can be computed efficiently. By exploring the properties of binary random parameters, the proposed reformulation techniques are extended to DRTSP with binary uncertainty. The main results in this paper are projected into the same decision space as conventional two-stage stochastic programs and deliver straightforward interpretable results of robustness. The proposed tractable results are shown to be sharp through complexity analysis. One possible future direction is that one might extend the proposed reformulation techniques for distributionally robust multi-stage stochastic programs with $\infty-$Wasserstein ambiguity set and derive tractable and intractable results.

\section*{Acknowledgments}

The author would like to thank Dr. Zhi Chen (City University of Hong Kong) for his insightful comments about $\infty-$Wasserstein ambiguity set.

\bibliography{Reference}

%
%
%
%

\end{document}